\documentclass{amsart}[1996/10/24]

\usepackage{graphicx}

\newtheorem{theorem}{Theorem}[section]

\theoremstyle{definition}
\newtheorem{definition}[theorem]{Definition}

\theoremstyle{plain}
\newtheorem{assumption}[theorem]{Assumption}

\theoremstyle{remark}
\newtheorem{remark}[theorem]{Remark}

\numberwithin{equation}{section}

\usepackage[]{hyperref}

\usepackage[oldcommands,linesnumbered,ruled,norelsize]{algorithm2e}

\usepackage{layouts}
\usepackage[table]{xcolor}
\definecolor{lightgray}{gray}{0.9}

\usepackage{enumerate}
\usepackage{pdflscape}

\usepackage{caption}
\usepackage{subcaption}

\newcommand{\dd}{\textrm{d}}
\newcommand{\y}{\textbf{y}}
\newcommand{\x}{\textbf{x}}
\newcommand{\maxSV}{\tilde{M}}

\newcommand{\R}{\mathbb{R}}
\newcommand{\N}{\mathbb{N}}
\newcommand{\Z}{\mathbb{Z}}

\newcommand{\supp}[1]{\mathrm{supp}(#1)} 
\newcommand{\length}[1]{\mathrm{length}(#1)} 

\DeclareMathOperator{\TS}{\mathtt{TS}}
\DeclareMathOperator{\TD}{\mathtt{TD}}
\DeclareMathOperator{\TP}{\mathtt{TP}}
\DeclareMathOperator{\aTD}{\mathtt{aTD}}
\DeclareMathOperator{\aTP}{\mathtt{aTP}}

\DeclareMathOperator{\isoTD}{\mathtt{isoTD}}
\DeclareMathOperator{\isoTP}{\mathtt{isoTP}}
\DeclareMathOperator{\TDFC}{\mathtt{TD-FC}}
\DeclareMathOperator{\bestM}{\mathtt{bestM}}

\DeclareMathOperator{\PtC}{\mathtt{PtC}}

\newcommand{\wigner}[3]{\begin{pmatrix} #1 & #2 & #3 \\ 0 & 0 & 0 \end{pmatrix}}
\newcommand{\wignersmall}[3]{\left(\begin{smallmatrix}#1 & #2 & #3 \\ 0 & 0 & 0 \end{smallmatrix}\right)}

\newcommand{\ind}[1]{\textbf{1}_{\{#1\}}}
\newcommand{\mind}[2]{\underset{\scriptstyle{#2}}{\scriptstyle{#1}}}

\begin{document}

\title[On efficient construction of stochastic moment matrices]{On efficient construction of stochastic moment matrices}

\author{Harri Hakula}
\address{Department of Mathematics and Systems Analysis, Aalto University, FI-00076 Aalto, Finland}
\curraddr{}
\email{harri.hakula@aalto.fi}

\author{Matti Leinonen}
\address{Department of Mathematics and Systems Analysis, Aalto University, FI-00076 Aalto, Finland}
\email{matti.leinonen@aalto.fi}
\thanks{This work was supported by the Finnish Doctoral Programme in Computational Sciences FICS and by the Academy of Finland (decision 267789)}


\date{}

\dedicatory{}

\begin{abstract}

We consider the construction of the stochastic moment matrices that appear in the typical elliptic diffusion problem considered in the setting of stochastic Galerkin finite element method (sGFEM).
Algorithms for the efficient construction of the stochastic moment matrices are presented for certain combinations of affine/non-affine diffusion coefficients and multivariate polynomial spaces.
We report the performance of various standard polynomial spaces for three different non-affine diffusion coefficients in a one-dimensional spatial setting and compare observed Legendre coefficient convergence rates to theoretical results.
\end{abstract}

\maketitle

\section{Introduction}
\label{sec:introduction}
Continuing progress of efficient deterministic numerical solvers for parametric partial 
differential equation models requires advances in several areas, including,
the parametrization of the uncertain and spatially inhomogeneous input,
the mathematical analysis of parametric differential equations,
and the actual development of the numerical solvers.
Advanced engineering applications are emerging as well. For instance, in
\cite{Yos15} it is noted that
more realistic and flexible models for stochastic input (SI) are needed.
This work extends computationally efficient modeling options in the context of 
stochastic Galerkin finite element method (sGFEM) with a tensor basis, focusing on the specific 
problem of construction of the associated moment matrices.
The structure of the induced graph of these matrices, and thus computational
complexity both in storage and time, depends on the chosen model
of stochastic input. This problem has been studied in the special case of
affine stochastic input \cite{Bieri09a,Bieri09b,Bieri09c}.
In this paper, the construction algorithms are extended to cover a much larger class of SI,
the non-affine models.
These ideas were motivated by the
conductivity reconstruction problems where the non-affine stochastic input arises naturally
via projection of the probability density function of the truncated 
multinormal distribution onto multivariate Legendre polynomial basis~\cite{Hyvonen14}.
Under very general assumptions, the computational complexity of the generalized algorithms
is asymptotically optimal. 
The numerical results further confirm the correctness
of the resulting systems. 

Let us introduce the model problem and state the main result more precisely.
We consider the following model elliptic diffusion problem:
Let $(\Omega,\Sigma,P)$ be a probability space.
Find random field $u \in L_P^2(\Omega, H_0^1(D))$ such that
\begin{align}
\label{equ:modelproblem}
\left\{
  \begin{array}{ll}
    -\nabla \cdot (a(\omega,\x) \nabla u(\omega,\x)) = f(\x) &\textrm{ in } D,\\
    u(\omega,\x) = 0 &\textrm{ on } \partial D
  \end{array}
\right.
\end{align}
holds $P$-almost surely for a load $f \in L^2(D)$ and 
a given strictly positive diffusion coefficient $a \in L^\infty(\Omega \times D)$ with lower and upper bounds 
$a_{\textrm{min}}, a_{\textrm{max}}$
such that
\begin{align}
\label{equ:assumptionona}
P\left(\omega \in \Omega \ : \ 0 < a_{\textrm{min}} \leq 
\underset{\x \in D}{{\rm ess}\,\! \inf}\, a(\omega,\x) \leq
\underset{\x \in D}{{\rm ess}\,\! \sup}\, a(\omega,\x) \leq  a_{\textrm{max}}\right) = 1.
\end{align}
Moreover, the physical domain $D$ is a bounded domain in $\R^d$, where $d\in\N$, with a smooth enough boundary and $L_P^2(\Omega, H_0^1(D))$ is a Bochner space; see, e.g., \cite{Schwab11a} and the references therein for more information.
The stochastic input, i.e., the diffusion coefficient, in the model problem is usually assumed to be given in the following \emph{affine} form:
\begin{definition}[Affine diffusion coefficient]
The affine diffusion coefficient is defined as
\begin{align}
\label{equ:klexpansion}
a(\omega,\x) = a_0(\x) + \sum_{m \geq 1} a_m(\x)Y_m(\omega),
\end{align}
where $\{a_m\}_{m \geq 0}$ are some suitable spatial functions and $\{Y_m\}_{m \geq 1}$ is a family of random variables.
This affine form is often obtained by using the Karhunen--Lo\`eve expansion; see, e.g., \cite{Adler07}.
\end{definition}

As noted above, in practical applications we might encounter \emph{non-affine} diffusion coefficients, and for example, in \cite{Chkifa15} it is stated that many practically relevant parametric PDEs are nonlinear and depend on the parameters $\{Y_m\}_{m \geq 1}$ in a non-affine manner.
Formally, we define the relevant concepts as follows:
\begin{definition}[Support of a multi-index]
\label{def:support}
The support of a multi-index $\eta\in\mathbb{N}_{0}^\infty$ is defined as
$
\supp{\eta} = \{ m\in\mathbb{N} \ | \ \eta_m\neq 0\}.
$
\end{definition}

\begin{definition}[Finitely supported multi-indices]
The set of finitely supported multi-indices $(\mathbb{N}_{0}^\infty)_c$ is defined by
\[
(\mathbb{N}_{0}^\infty)_c = \{ \eta\in\mathbb{N}_{0}^\infty\ | \ |\supp{\eta}|<\infty\} \subset \mathbb{N}_{0}^\infty,
\]
where we use $|\supp{\eta}|$ to denote the cardinality of $\supp{\eta}$.
We define the set $(\Z^\infty)_c$ in a similar way.
\end{definition}
\begin{definition}[Non-affine diffusion coefficient]
The non-affine diffusion coefficient is given by
\begin{align}
\label{equ:nonaffinekl}
a(\omega,\x) = a_0(\x) + \sum_{\mu \in \Xi} a_\mu(\x)Y^\mu(\omega),
\end{align}
where $\Xi$ is a subset of finitely supported multi-indices,
$a_0$ and $\{a_\mu\}_{\mu\in\Xi}$ are some suitable spatial functions,
and 
\begin{align*}
Y^\mu(\omega):=
\prod_{m=1}^\infty Y_m^{\mu_m}(\omega)= 
\prod_{m\in\supp{\mu}} Y_m^{\mu_m}(\omega).
\end{align*}
\end{definition}

For example, 
the (typical) instance 
\begin{align*}
a(\omega,\x) = a_0(\x) + \left(\sum_{m \geq 1} a_m(\x)Y_m(\omega)\right)^2
\end{align*}
considered in \cite{Chkifa15} fits the non-affine form \eqref{equ:nonaffinekl}.
In this paper, we refer with \emph{non-affine} to the form \eqref{equ:nonaffinekl} and we do not consider other non-affine type diffusion coefficients, e.g., log-normal,  unless explicitly stated.
Notice that the non-affine case includes the affine case;
by selecting 
$\Xi = 
\{\mu \in (\mathbb{N}_{0}^\infty)_c\,|\,\mu = {\mathrm e}_n~\textrm{for some}~n\in\N\}$,
where ${\mathrm e}_n$ denotes the $n$th Euclidean basis vector of $\R^\infty$,
 the non-affine case reduces to the affine case.
Here, as in the following, we have used the convention $0^0=1$.

A standard approach is to transform the model problem \eqref{equ:modelproblem} into a parametric deterministic form by first assuming the family $\{Y_m\}_{m \geq 1}:\Omega\to\R$ to be \emph{mutually independent} with ranges $\Gamma_m$ and associating each $Y_m$ with a complete probability space $(\Omega_m, \Sigma_m, P_m)$, where the probability measure $P_m$ admits a probability density function $\rho_m:\Gamma_m\to [0,\infty)$ such that 
\[
\dd P_{m}(\omega) = \rho_m(y_m) \dd y_m, \quad y_m\in\Gamma_m,
\]
and the $\sigma$-algebra $\Sigma_m$ is assumed to be a subset of the Borel sets of $\Gamma_m$. Finally, it is assumed that the stochastic input data is {\em finite}, i.e., in the affine case we assume that there exists $\maxSV<\infty$ such that $Y_m \equiv 0$, when $m > \maxSV$, which essentially truncates the series in \eqref{equ:klexpansion} after $\maxSV$ terms, and in the non-affine case~\eqref{equ:nonaffinekl} we assume that the set $\Xi$ has a finite number of elements and $\maxSV$ as in the affine case exists.

After these transformations, the parametric deterministic weak formulation of \eqref{equ:modelproblem} is to find 
$u \in L_\rho^2(\Gamma, H_0^1(D))$
that satisfies 
\begin{align}
\label{equ:1Dmodelweakform}
\int_\Gamma \int_D a(\y,\x) \nabla u(\y,\x) \cdot \nabla v(\y,\x) \rho(\y) \dd\x \dd\y = 
\int_\Gamma \int_D f(\x) v(\y,\x) \rho(\y) \dd \x \dd \y
\end{align}
for all $v \in L_\rho^2(\Gamma, H_0^1(D))$, where 
$
\y := (y_1,y_2,\ldots)\in\Gamma:= \Gamma_1\times\Gamma_2\times\ldots
$
and
$
\rho(\y)\dd \y := \prod_{m\geq 1}\rho_m(y_m)\dd y_m.
$
The unique solvability of \eqref{equ:1Dmodelweakform} follows from the Lax--Milgram lemma.

By discretization, the weak formulation~\eqref{equ:1Dmodelweakform} reduces to the following linear system of equations in the affine case \eqref{equ:klexpansion}:
\begin{align}
\label{equ:reslinearsystem1}
\left(G_0 \otimes A_0 + \sum_{m=1}^{\maxSV} G_m \otimes A_m\right)\mathbf{u} = g_0 \otimes f_0,
\end{align}
where 
$\{A_m\}_{m=0}^{\maxSV}$ are standard finite element (FEM) matrices,
$f_0$ is a standard load vector,
$\{G_m\}_{m\geq0}$ are stochastic moment matrices, and
$g_0$ is a stochastic load vector; see \cite{Bieri09c} for details.

For the stochastic load vector and moment matrices, we recall that each multi-index $\eta\in (\mathbb{N}_{0}^\infty)_c$ determines a multivariate polynomial $\Phi_\eta(\y)$.
\begin{definition}[Multivariate polynomial]
\label{def:multivariatepolynomial}
Let $\eta\in (\mathbb{N}_{0}^\infty)_c$. The multivariate polynomial $\Phi_\eta(\y)$, also called ``chaos polynomial'', is defined as 
\begin{align}
\label{equ:multivariatepolynomial}
\Phi_\eta(\y) = 
\prod_{m=1}^\infty \phi_{\eta_m}(y_m) = 
\prod_{m\in\supp{\eta}} \phi_{\eta_m}(y_m),
\end{align}
where $\{\phi_m\}_{m \geq 0}$ is a suitably orthonormalized (univariate) polynomial sequence with $\phi_0 \equiv 1$. 
\end{definition}
In this paper, we assume that $\{\phi_m\}_{m \geq 0}$ is either the Legendre or the (probabilists') Hermite polynomial family depending on the family $\{Y_m\}_{m \geq 1}$ as those are typical choices in the sGFEM setting; see Appendix~\ref{sec:orthopolys} for the definition of the Legendre and Hermite polynomials.
We call the different cases as the Legendre case and the Hermite case, respectively.

\begin{remark}
\label{rem:hermiteremark}
We consider the Hermite case mainly for completeness as the condition \eqref{equ:assumptionona} combined with the non-affine form~\eqref{equ:nonaffinekl} has to be relaxed for the use of normally distributed random variables $\{Y_m\}_{m \geq 1}$; we refer to \cite{Babuska10, BieriThesis09,Charrier12,Gittelson10} for possible relaxations of condition \eqref{equ:assumptionona}.
We content ourselves by providing the formulas for the stochastic moment matrices in the Hermite case, and in the numerical experiments of Section~\ref{sec:ne}, we consider only the Legendre case.
\end{remark}

In the affine case, the stochastic load vector is given by
\begin{align}
\big[g_0\big]_{\alpha} := \int_\Gamma \Phi_\alpha(\y) \rho(\y) \dd \y
\end{align}
and the elements of the stochastic moment matrices are 
\begin{equation}
\begin{aligned}
\label{equ:smmdef}
\big[G_0\big]_{\alpha,\beta} &:= \int_\Gamma \Phi_\alpha(\y) \Phi_\beta(\y) \rho(\y) \dd \y\quad \textrm{and}\\
\big[G_m\big]_{\alpha,\beta} &:= \int_\Gamma y_m \Phi_\alpha(\y) \Phi_\beta(\y) \rho(\y) \dd \y, \quad m\geq 1,
\end{aligned}
\end{equation}
where 
$\alpha, \beta \in \Lambda \subset (\mathbb{N}_{0}^\infty)_c$. 
The index set $\Lambda$ is finite and consists of suitably chosen finitely supported multi-indices; see \cite{Schwab11a} and the references therein for more information. 
In an ideal case, the index set $\Lambda$ is chosen so that
\begin{align*}
\sum_{\mu \in \Lambda} u_{\mu}( \x)\Phi_{\mu}(\y),
\end{align*}
where $u_{\mu} := \int_\Gamma u(\y,  \cdot) \Phi_\mu(\y) \rho(\y) \dd \y \in H_0^1(D)$,
gives a good approximation to the solution of~\eqref{equ:1Dmodelweakform}; in practice, the index set is often chosen more or less arbitrarily.
It is known that the stochastic moment matrices $\{G_m\}_{m\geq1}$ exhibit a nontrivial sparsity pattern~\cite{Bieri09c}.
($G_0$ is a diagonal matrix.)

Efficient construction of the stochastic moment matrices~\eqref{equ:smmdef} is crucial in the sGFEM setting 
as it easily becomes one of the bottlenecks for the method when the number of multi-indices in $\Lambda$ increases. One solution is to build the index set $\Lambda$ so that an efficient construction of the moment matrices is possible; Bieri et al. have given such an algorithm in \cite{Bieri09a} for certain index set types.

In the non-affine case \eqref{equ:nonaffinekl}, the resulting linear system is (cf.~\eqref{equ:reslinearsystem1})
\begin{align}
\label{equ:reslinearsystem2}
\left(G_0 \otimes A_0 + \sum_{\mu\in\Xi} G^\mu \otimes A_\mu\right)\mathbf{u} = g_0 \otimes f_0,
\end{align}
where 
$A_0$ and $\{A_\mu\}_{\mu\in\Xi}$ are standard FEM matrices and
the stochastic moment matrices $\{G^\mu\}_{\mu\in\Xi}$ are given by
\begin{align}
\label{equ:SMM}
\big[G^\mu\big]_{\alpha,\beta} 
&:= \int_\Gamma
\left(\prod_{m=1}^\infty y_m^{\mu_m}\right) 
\Phi_\alpha(\y) 
\Phi_\beta(\y) 
\rho(\y) 
\dd \y
\end{align}
with $\alpha$ and $\beta$ as above. With this notation, the matrix $G_m$, where $m \in \mathbb{N}$, is given by $G^{{\mathrm e}_m}$
and $G_0$ is given with $\mu=(0,0,\ldots)$. 
As above, these moment matrices employ a nontrivial sparsity pattern.

The efficient construction of the matrices $\{G^\mu\}_{\mu\in\Xi}$ is inspired by
\cite[Algorithm~4.13]{Bieri09a} for $\{G_m\}_{m\geq 0}$.
We present a non-recursive version of the original
and generalize it for $\{G^\mu\}_{\mu\in\Xi}$. 
This leads to a novel setup for updating the necessary auxiliary information in contrast
to~\cite{Bieri09a} (and its research report version~\cite{Bieri09b}).
This allows us to form the stochastic moment matrices for the model problem \eqref{equ:modelproblem} efficiently when the stochastic input is given in a non-affine form \eqref{equ:nonaffinekl}. 

The computational complexity depends on the set of multi-indices $\Lambda$ and the chosen model of the stochastic input.
However, the latter part has finite complexity in the sense that it has to be fully resolved and thus
it is not asymptotic. Therefore, the computational complexity is ultimately determined by the set of multi-indices only.
However, the situation is similar to all sparse versus full considerations. For a sufficiently small set
of multi-indices it is possible that the stochastic input determines the \textit{observed} complexity.
Asymptotically, we get the complexity estimate $O(|\Lambda|^{1+\epsilon})$, where $0 < \epsilon \leq 1$, 
and in many practical cases $\epsilon \ll 1$.

The rest of the paper is organized as follows.
In Section~\ref{sec:preliminaries}, we give definitions for the multi-index sets $\Lambda$ considered in this paper and fix some terminology and conventions related to multi-indices.
Section~\ref{sec:smm} gives steps for constructing the stochastic moment matrices $\{G^\mu\}_{\mu\in\Xi}$ efficiently. 
Algorithms for the construction of the multi-indices and neighbour matrices required for $\{G^\mu\}_{\mu\in\Xi}$ are presented in Section~\ref{sec:cm}, and our numerical experiments are presented in Section~\ref{sec:ne}.
In the numerical experiments, we compare the performance of different multi-index sets and compare observed Legendre coefficient convergence rates to theoretical results.
Finally, we conclude with discussion in Section~\ref{sec:co}.
 In Appendix~\ref{sec:orthopolys}, we recall some essential properties of the Legendre and Hermite polynomials which were employed in Section~\ref{sec:smm}, and Appendices~\ref{App:A}--\ref{App:D} contain the actual pseudocodes for the algorithms presented in Section~\ref{sec:cm}.

\section{Preliminaries}
\label{sec:preliminaries}

To improve the readability of this paper, we start by introducing the six multi-index sets considered in this paper: 
\begin{enumerate}[1.]
\item isotropic tensor product ($\isoTP$) index set,
\item isotropic total degree ($\isoTD$) index set,
\item anisotropic tensor product ($\aTP$) index set,
\item anisotropic total degree ($\aTD$) index set,
\item threshold sequence ($\TS$) index set, and
\item the best $M$-terms ($\bestM$) index set.
\end{enumerate}

Standard choices for the index set $\Lambda$ are 
the 
$\isoTP$ and 
$\isoTD$ index sets. Selecting the optimal index set a priori in the sGFEM setting is an active research question; see, e.g., \cite{Beck12, Bieri09a, Back11} and the references therein for more information.

\begin{definition}[Isotropic tensor product index set]
Let $N \in \N$ and $K \in \N_0$.
The 
$\isoTP$ index set is given by
\begin{align}
\label{equ:TPcond}
\isoTP(N,K) = \left\{\alpha\in (\mathbb{N}_{0}^\infty)_c~|~\max_{n=1,\ldots,N}{\alpha_n\leq K}; \alpha_n = 0, n>N  \right\}.
\end{align}
\end{definition}

\begin{definition}[Isotropic total degree index set]
Let $N \in \N$ and $K \in \N_0$. 
The 
$\isoTD$ index set is defined as
\begin{align}
\label{equ:TDcond}
\isoTD(N,K) = \left\{\alpha\in (\mathbb{N}_{0}^\infty)_c~|~\sum_{n=1}^{N}\alpha_n\leq K; \alpha_n = 0, n>N  \right\}.
\end{align}
\end{definition}
The isotropic index sets are often preferred due to their easiness of construction.
However, the cardinalities of 
the $\isoTP$ and $\isoTD$ 
index sets are $(K+1)^N$ and ${N+K \choose K}$, respectively,
and hence the isotropic index sets 
grow rapidly when either $N$ or $K$ is increased.
Therefore, as the $\isoTD$ index set grows slower than the $\isoTP$ index set,
the $\isoTD$ index set is often preferred over the $\isoTP$ index set.

The isotropic index sets give equal weight to the $N$ dimensions, and as in some applications we might want to prefer some of the dimensions, the {\em anisotropic} versions of the isotropic index sets are also used.

\begin{definition}[Anisotropic tensor product index set]
Let $N \in \N, K \in \N_0,$ and $\mu = (\mu_1,\ldots,\mu_N)\in\R_+^N$ be a vector of positive weights. 
The {\em anisotropic tensor product} ($\aTP$) index set is given by
\begin{align}
\label{equ:TPcondaniso}
\aTP(N,K,\mu) = \left\{\alpha\in (\mathbb{N}_{0}^\infty)_c~|~\max_{n=1,\ldots,N}{\mu_n\alpha_n\leq \mu_{\min}K}; \alpha_n = 0, n>N \right\}.
\end{align}
Here and in the following, $\R_+:=\{x\in\R\,|\,x>0\}$ and $\mu_{\min} := \min(\mu)$.
Similarly, we define $\mu_{\max} = \max(\mu)$.
\end{definition}

\begin{definition}[Anisotropic total degree index set]
Let $N \in \N, K \in \N_0,$ and $\mu = (\mu_1,\ldots,\mu_N)\in\R_+^N$ be a vector of positive weights. 
The {\em anisotropic total degree} ($\aTD$) index set is defined as
\begin{align}
\label{equ:TDcondaniso}
\aTD(N,K,\mu) = \left\{\alpha\in (\mathbb{N}_{0}^\infty)_c~|~\sum_{n=1}^{N}\mu_n\alpha_n\leq \mu_{\min}K; \alpha_n = 0, n>N \right\}.
\end{align}
\end{definition}
The weight vector $\mu$ is used to give different weights to the $N$ dimensions, and thus the cardinalities of the anisotropic index sets depend heavily on the chosen weight vector.
Without loss of generality, we assume that the weight vector $\mu$ used in the $\aTP$ and $\aTD$ index sets is non-decreasing,
i.e.,
$\mu_{\min} = \mu_1 \leq \mu_2 \leq \cdots \leq \mu_N$.
(If it is not the case, we permute the $N$ dimensions according to the weight vector.)
We refer with $\TP$ and $\TD$ to both anisotropic and isotropic versions of the tensor product and total degree index sets, respectively. Notice, that by using a constant weight vector, the anisotropic index sets reduce to the corresponding isotropic index sets.

Closely related to the $\TD$ spaces is, what we call, the {\em threshold sequence} ($\TS$) index set.
\begin{definition}[Threshold sequence index set]
Let $0 < \varepsilon \leq 1$ and $\mu$ be a non-increasing sequence of nonnegative real numbers bounded above by one and tending to zero, i.e., 
\[
\mu \in c_0 := \left\{ \mu = ( \mu_1,\mu_2,\ldots )\in \R_0^\infty~|~\lim_{m\to\infty} \mu_m = 0 \quad {\rm and} \quad 1 > \mu_1 \geq \mu_2 \geq \dots \right\},
\]
where $\R_0 := \{x \in \R \,|\,x \geq 0\}$.
The $\TS$ index set is given by
\begin{align}
\label{equ:TScond}
\TS(\mu,\varepsilon) = \left\{ \alpha\in(\mathbb{N}_{0}^\infty)_c ~|~ \mu^\alpha := \prod_{m\in\N}\mu_m^{\alpha_m} = \prod_{m\in\supp\alpha}\mu_m^{\alpha_m} \geq \varepsilon  \right\}.
\end{align}
In \cite{Bieri09a}, essentially the same index set as $\TS(\mu,\varepsilon)$ is denoted with $\Lambda_\varepsilon(\mu)$. (In \cite{Bieri09a}, $\R_0^\infty$ is replaced with $\R_+^\infty$.) 
\end{definition}

By noticing that the $\aTD$ set condition 
\[
\sum_{n=1}^N \mu_n \alpha_n \leq \mu_{\min} K
\]
can be written as
\[
\prod_{n=1}^N\tilde\mu_n^{\alpha_n} \geq \varepsilon
\]
by defining
\begin{align*}
\tilde\mu = \left(\exp\left(-\frac{\mu_1}{\mu_{\min}}\right),\ldots,\exp\left(-\frac{\mu_N}{\mu_{\min}}\right)\right)
\,\textrm{ and }\,
\varepsilon = \exp(-K),
\end{align*}
an $\aTD(N,K,\mu)$ index set can be considered as a $\TS(\tilde\mu,\varepsilon)$ index set by permuting and padding $\tilde\mu$ so that the vector is non-increasing and belongs to $\R_0^\infty$.
In most cases, we are interested to construct an index set with less than some maximum number of multi-indices and not interested in the actual parameters used to construct the index set.
Hence, in the anisotropic cases it is often enough to define the weight vector $\mu$ up to a positive constant $\frac{c}{\mu_\textrm{min}}$, where $c\in\R_+$, by using the vector
\[
g = \left(c\frac{\mu_1}{\mu_{\min}},\ldots,c\frac{\mu_N}{\mu_{\min}}\right) \in\R_+^N.
\]
For example, the $\aTD$ index set condition \eqref{equ:TDcondaniso}
can be written as 
\[
\exp\left(-\sum_{n=1}^N g_n \alpha_n\right)
\geq \exp(-cK)
\]
and by considering $\exp(-cK)$ as a real number between zero and one independent of $c$ and $K$, we can control the number of elements in the index set by varying the value of $\exp(-cK)$.
To summarize:
\begin{itemize}
\item The $\aTD$ index set can be considered as the $\TS$ index set.
\item The $\TS$ index set allows us in many practical cases to define the weight vector $\mu$ in the $\TD$ index set up to a constant.
\end{itemize}
Similarly in the $\aTP$ case, it is often enough to define the weight vector $\mu$ up to a positive constant.

Finally, the $\bestM$ index set is obtained from an existing index set.
\begin{definition}[Best $M$-terms index set]
The $\bestM$ index set is obtained from an existing index set $\Lambda$ by associating each element of $\Lambda$ with a real valued coefficient, usually the norm of $u_\mu$, after which the elements of $\Lambda$ are sorted in a non-increasing order based on the corresponding coefficients. The first $M$-indices from the resulting ordered sequence form the $\bestM$ index set.
\end{definition}
Using the $\bestM$ index set in practice is often impractical as to construct the index set, i.e., to obtain the norms of $u_\mu$, we first compute an sGFEM solution corresponding to a large index set and there is often no benefit in considering smaller index sets than the large index set already used to construct the sGFEM solution. 
Hence, we focus on the other index sets when discussing the obtained results and use the $\bestM$ index set only as a reference index set. 
However, there are at least two cases (not considered in this paper) where it might be useful to actually consider the $\bestM$ index sets.
First, evaluating the sGFEM solution might be expensive and one option to reduce the evaluation cost is to reduce the size of the index set using the $\bestM$ index set.
This technique will, in general, increase the approximation error but by dropping the multi-indices corresponding to the smallest $u_\mu$, we expect the approximation error to increase only slightly.
Second, if we are using an adaptive scheme, like reservoir sampling~\cite{Vitter85}, to construct the index set, it may be useful to consider $\bestM$ index sets; constructing the index set using randomized algorithms is left for future studies.

Let us close this section with some terminology and conventions related to multi-indices.
In this paper, the addition, subtraction, and product of multi-indices is carried out componentwise and, if necessary, the elements of $(\N_0^\infty)_c$ are considered as elements of $(\Z^\infty)_c$.
Moreover, we use the following definitions for the length and parenthood condition of a multi-index.

\begin{definition}[Length of a multi-index]
Let $\alpha \in (\mathbb{N}_{0}^\infty)_c$. The length of $\alpha$ is defined as 
\[
\length{\alpha}=\max(\supp{\alpha} \cup \{0\}),
\]
where the support of a multi-index was defined in Definition~\ref{def:support}.
\end{definition}

\begin{definition}[Parenthood condition of a multi-index]
\label{def:parenthoodcondition}
Let $\alpha$ and $\beta$ be two multi-indices. We call $\alpha$ the parent of $\beta$ or $\beta$ the child of $\alpha$ if there exists $n\in \N$ greater or equal to $\length{\alpha}$ such that $\beta_m = \alpha_m + \delta_{mn}$ for all $m \in \N$, where $\delta_{mn}$ is the Kronecker's delta.
\end{definition}

In the algorithms presented in this paper, we store the parenthood relations between the multi-indices of a index set $\Lambda$ in a variable denoted with $\PtC$ in such a way that the children of $\alpha\in\Lambda$ are given in the vector $\PtC[\alpha]$. (Actually, the children of $\alpha$ are given in the vector $\PtC[a]$, where $a\in\N$ is an index corresponding to $\alpha$, but for notational reasons we write $\PtC[\alpha]$ in the text and use the correct form $\PtC[a]$ in the actual pseudocodes.)
Moreover, each multi-index $\alpha \in (\mathbb{N}_{0}^\infty)_c$ or vector $\alpha \in (\Z^\infty)_c$ is stored in the presented pseudocodes in a sparse format 
\[
\alpha \equiv \{(m,\alpha_m)\,|\,\alpha_m \neq 0 \}
\]
and each multi-index $\alpha \in \TS(\mu,\varepsilon)$ is stored as
\[
\alpha \equiv \{\{(m,\alpha_m)\,|\,\alpha_m \neq 0 \}, \mu^\alpha \};
\]
for example, the multi-index $(0,1,2,0,\ldots)$ would be stored in $\TS(\mu,\varepsilon)$ as
\[
\{\{(2,1),(3,2)\},\mu_2^1 \mu_3^2\}
\]
if it belongs to the set, i.e., if $\mu_2^1 \mu_3^2 \geq \varepsilon$.
In the pseudocodes, we denote with $\alpha[m]$, where $m\in\N$, the $m$th pair in the sparse format of $\alpha$. Hence, in general, $\alpha[m]$ does not necessarily correspond to $\alpha_m$. Moreover, we denote with $\alpha[-1]$ the last pair in the sparse format of $\alpha$, and we use short-circuit evaluation of Boolean operators, i.e., the second argument of a Boolean expression is evaluated only if the first argument does not suffice to determine the value of the expression.

\section{Construction of stochastic moment matrices}
\label{sec:smm}

Combining Definition~\ref{def:multivariatepolynomial} to \eqref{equ:SMM}, we can write the element of $G^\mu$ as 
\begin{align*}
\big[G^\mu\big]_{\alpha,\beta} 
&= \int_\Gamma
\left(\prod_{m=1}^\infty y_m^{\mu_m}\right) 
\Phi_\alpha(\y) 
\Phi_\beta(\y) 
\rho(\y) 
\dd \y \\
&= \int_\Gamma
\left( \prod_{m=1}^\infty y_m^{\mu_m} \right) 
\left( \prod_{m=1}^\infty \phi_{\alpha_m}(y_m) \right) 
\left( \prod_{m=1}^\infty \phi_{\beta_m}(y_m) \right) 
\rho(\y) 
\dd \y \\
&= \prod_{m\in\supp{\mu + \alpha + \beta}} 
\big[G_m^{\mu_m}\big]_{\alpha_m+1, \beta_m+1},
\end{align*}
where 
$\alpha$ and $\beta$ are as before
and 
the matrix $G_m^{\mu_m}$ is given by 
(here, $\alpha_m+1$ and $\beta_m+1$ are used as general indices and are independent of $m$)
\[
\big[G_m^{\mu_m}\big]_{\alpha_m+1, \beta_m+1} = 
\int_{\Gamma_m}
y_m^{\mu_m} 
\phi_{\alpha_m}(y_m)
\phi_{\beta_m}(y_m)
\rho_m(y_m) \dd y_m.
\]

We do the following assumptions in the Legendre and Hermite cases, respectively:
\begin{assumption}
\label{ass:legendre}
In the Legendre case, we assume that 
\[
\Gamma_m = [-1,1] := \Gamma_0\quad \textrm{ and }\quad \rho_m(y) = \frac{1}{2} := \rho_0(y)\] for all $m \geq 1$.
\end{assumption}

\begin{assumption}
\label{ass:hermite}
In the Hermite case, we assume that 
\[
\Gamma_m = \R := \Gamma_0\quad \textrm{ and }\quad \rho_m(y) = \exp(-y^2/2)/\sqrt{2\pi} := \rho_0(y)
\]
for all $m \geq 1$. 
(Recall remark~\ref{rem:hermiteremark} close to the end of Section~\ref{sec:introduction}.) 
\end{assumption}

Due to Assumptions~\ref{ass:legendre} and \ref{ass:hermite},
$G_m^{k} = G_n^{k}$ for all $k\in\N_0$ and $m,n \geq 1$ (each of the stochastic dimensions has the same range and probability density), and hence it is enough to consider only the matrix $K^k$ given by (here again $l+1$ and $m+1$ are general indices)
\[
\big[K^k\big]_{l+1, m+1} := 
\int_{\Gamma_0}
y^{k} 
\phi_{l}(y)
\phi_{m}(y)
\rho_0(y) \dd y,
\]
where $k,l$, and $m$ are nonnegative integers,
as we can write 
\begin{align}
\label{equ:Gdef}
\big[G^\mu\big]_{\alpha,\beta} = 
\prod_{m\in\supp{\mu + \alpha + \beta}}
\big[K^{\mu_m}\big]_{\alpha_m+1, \beta_m+1}.
\end{align}
By recalling that we are using \emph{orthonormal} multivariate polynomials, $K^0$ is an identity matrix, and therefore $G^\mu$ can be computed as 
\begin{align}
\label{equ:Gformula}
\big[G^\mu\big]_{\alpha,\beta} = 
\left\{
  \begin{array}{ll}
\underset{m\in\supp{\mu}}\prod\big[K^{\mu_m}\big]_{\alpha_m+1, \beta_m+1},&\textrm{when } \big[G^\mu\big]_{\alpha,\beta} \neq 0,\\
    0, &\textrm{otherwise}.
  \end{array}
\right.
\end{align}
To use \eqref{equ:Gformula}, we need to locate the nonzero elements in $G^\mu$, which amounts to understanding the structure of $K^k$ as is evident from~\eqref{equ:Gdef}.

Using the triple integral and inversion formulas for the Legendre and Hermite polynomials from Appendix~\ref{sec:orthopolys}, we can locate and compute the nonzero entries in the matrix $K^k$.

\begin{theorem}
\label{thm:KforLegendre}
In the Legendre case, we have
\begin{align}
\label{equ:KforLegendre}
\big[K^k\big]_{l+1, m+1} = \sum_{n=0}^k l_n^k L(l,m,n),
\end{align}
where 
\[
l^k_n = 
\ind{k-n~\textrm{is even}}
{k \choose n} 
n!\,
\sqrt{2n+1}
\frac{(k-n-1)!!}{(k+n+1)!!},
\quad 0 \leq n \leq k,
\]
and
\[
L(l,m,n) = \sqrt{(2l+1)(2m+1)(2n+1)}\wigner{l}{m}{n}^2.
\]
\end{theorem}
\begin{proof}
\begin{align*}
\begin{split}
\big[K^k\big]_{l+1, m+1} 
&= \int_{-1}^1 y^k L_l(y) L_m(y) \rho_0(y) \dd y \\
&=\sum_{n=0}^k l_n^k \int_{-1}^1 L_n(y) L_l(y) L_m(y) \rho_0(y) \dd y
=\sum_{n=0}^k l_n^k L(l,m,n),
\end{split}
\end{align*}
where we have used first Theorem~\ref{thm:inversioforLegendre} and then Theorem~\ref{thm:tintforLegendre}.
\end{proof}

Combining the selection rules of $l_n^k$ and $L(l,m,n)$ from \eqref{thm:inversioforLegendre} and \eqref{equ:wignerdef}, respectively, we know that the summand in \eqref{equ:KforLegendre} is nonzero when
\begin{itemize}
\item $k-n$ is even,
\item $l+m+n$ is even, and
\item $|l-m| \leq n \leq l+ m$, where $0 \leq n \leq k$.
\end{itemize}
From the first condition, we see that if $k$ is even, $n$ is also even and then the second condition imposes $l+m$ to be even as well. Hence, $l$ and $m$ have the same parity when $k$ is even.
Similarly, when $k$ is odd, $n$ is also odd, and as a result $l$ and $m$ have the opposite parity.
From the last condition, we see that $K^k$ is a symmetric $(2k+1)$-diagonal matrix.
To sum up, $K^k$ is a symmetric $(2k+1)$-diagonal matrix with odd diagonals zero when $k$ is even and even diagonals zero when $k$ is odd. 
For example, the common case $K^1$ used in the affine case is given by~\cite{Bieri09c} (correcting a misprint) 
\begin{align*}
\big[K^1\big]_{l,m} = 
\left\{
  \begin{array}{ll}
    \frac{l}{\sqrt{(2l-1)(2l+1)}}, &m=l+1,\\
    \frac{(l-1)}{\sqrt{(2l-3)(2l-1)}}, &m=l-1,\\
    0, &\textrm{otherwise},
  \end{array}
\right.
\end{align*}
where $l,m \in \N$.
Similar results hold for the Hermite case.
\begin{theorem}
\label{thm:KforHermite}
In the Hermite case, we have
\begin{align}
\label{equ:KforHremite}
\big[K^k\big]_{l+1, m+1} =\sum_{n=0}^k h_n^k H(l,m,n),
\end{align}
where 
\[
h^k_n = 
\ind{k-n~\textrm{is even}}
{k \choose n} 
\sqrt{n!}\,
(k-n-1)!!,
\quad 0 \leq n \leq k,
\]
and 
\[
H(l,m,n) = \ind{2g~\textrm{is even}}\ind{|l-m|\leq n \leq l+m}\left[{l \choose {g-n} } {m \choose {g-n} } {n \choose {g-l} }\right]^\frac{1}{2}
\]
with $g = (l+m+n)/2$.
\end{theorem}
\begin{proof}
Similar to the Legendre case; use Theorems~\ref{thm:inversioforHermite} and \ref{thm:tintforHermite}.
\end{proof}

By noticing that the selection rules in the summand of Theorem~\ref{thm:KforHermite} are the same as in the Legendre case, $K^k$ has the same structure as in the Legendre case.
The common case $K^1$ is in the Hermite case
\begin{align*}
\big[K^1\big]_{l,m} = 
\left\{
  \begin{array}{ll}
    \sqrt{l}, &m=l+1,\\
    \sqrt{l-1}, &m=l-1,\\
    0, &\textrm{otherwise},
  \end{array}
\right.
\end{align*}
where $l,m \in \N$.

Using the structure of $K^k$ and \eqref{equ:Gdef}, we see that the element $\big[G^\mu\big]_{\alpha,\beta}$ is nonzero if the following condition holds for all $m \geq 1$: $|\alpha_m-\beta_m| \leq \mu_m$ and $\alpha_m$ and $\beta_m$ have the same (opposite) parity when $\mu_m$ is even (odd). This suggests the following for the Legendre and Hermite cases: Associate each multi-index $\mu \in (\mathbb{N}_{0}^\infty)_c$ with the following set of weights
\[
W(\mu) := \{\alpha\beta\,|\,\alpha \in P(\mu) \textrm{ and } \beta \in F(\mu) \} \subset (\Z^\infty)_c,
\]
where 
\begin{align*}
P(\mu) &:= \underset{m \geq 1}{\prod}
\left\{
  \begin{array}{ll}
    \{1\}, &\mu_m=0 \textrm{ or } \mu_n=0 \textrm{ for all } n < m,\\
    \{1,-1\}, &\textrm{otherwise},
  \end{array}
\right.
\end{align*}
and
\begin{align*}
F(\mu) &:= \prod_{m \geq 1}
\left\{
  \begin{array}{ll}
    \{0,2,\cdots,\mu_m\}, &\mu_m\textrm{ is even},\\
    \{1,3,\cdots,\mu_m\}, &\mu_m\textrm{ is odd}.
  \end{array}
\right.
\end{align*}
Now, the condition for the nonzero elements in $G^\mu$ can be given as:
\begin{align}
\label{equ:Gzerocond}
\big[G^\mu\big]_{\alpha,\beta} \neq 0 \Leftrightarrow
\big[{\textstyle \sum_{w \in W(\mu)}} N^w\big]_{\alpha,\beta} \neq 0 \Leftrightarrow
\big[S^\mu\big]_{\alpha,\beta} \neq 0,
\end{align}
where $N^w$ is the neighbour matrix given by
\begin{align}
\label{equ:neighbourmatrix}
\big[N^w\big]_{\alpha,\beta} := 
\left\{
  \begin{array}{ll}
    1, &\alpha \pm w = \beta,\\
    0, &\textrm{otherwise},
  \end{array}
\right.
\end{align}
and $S^\mu$ is the summed matrix
\[
S^\mu := \sum_{w \in W(\mu)}N^w.
\]
Finally, the stochastic moment matrices $\{G^\mu\}_{\mu\in\Xi}$ for finitely supported multi-index sets $\Xi$ and $\Lambda$ can be computed in the following way:
\begin{enumerate}
\item Construct the neighbour matrices $\{N^w\}$, where $w \in W_\Xi:=\underset{\mu\in\Xi}{\bigcup}W(\mu)$.
\item Compute the summed matrices $\{S^\mu\}_{\mu \in \Xi}$ using $\{N^w\}_{w \in W_\Xi}$.$\phantom{\underset{\mu\in\Xi}{\bigcup}}$
\item Construct the matrices $\{K^k\}$, where $k \in \left\{1,\ldots,\underset{\mu\in\Xi}{\max}(\max(\mu))\right\}$, using Theorem~\ref{thm:KforLegendre} or Theorem~\ref{thm:KforHermite}.$\phantom{\underset{\mu\in\Xi}{\bigcup}}$
\item Compute the matrices $\{G^\mu\}_{\mu\in\Xi}$ using \eqref{equ:Gformula} and \eqref{equ:Gzerocond}.
\end{enumerate}
Here, $N^w$, $S^\mu$, and $G^\mu$ are symmetric (sparse) matrices of dimension 
$|\Lambda|$,
and $K^k$ is a symmetric $(2k+1)$-diagonal matrix of dimension $\underset{\alpha\in\Lambda}{\max}(\max(\alpha))+1$.

The second, third, and fourth steps from above are straightforward to carry out.
However, the first step, i.e., constructing the neighbour matrices $\{N^w\}_{w \in W_\Xi}$, is a nontrivial task. 
Straightforward methods to construct the neighbour matrices such as using the definition \eqref{equ:neighbourmatrix} directly or writing $w=\pm(\alpha - \beta)$ and looping over $\alpha$ and $\beta$ result in algorithms where we essentially loop through $W_\Xi$ once for each $\{\alpha, \beta\}$ pair, i.e., the construction requires $\mathcal{O}(|W_\Xi| |\Lambda|^2)$ operations.
In the next section, we give an algorithm based on \cite{Bieri09b} which can be used to construct hierarchical multi-index sets so that constructing neighbour matrices for these index sets require $\mathcal{O}(|W_\Xi| |\Lambda|^{1+\epsilon})$ operations, where $0 < \epsilon \leq 1$; the exact value of $\epsilon$ depends on $\Xi$ and $\Lambda$ and is much less than one in typical instances.  
The speedup is obtained by constructing the multi-indices in $\Lambda$ in a such way that locating a multi-index in $\Lambda$ does not necessarily require us to loop over entire $\Lambda$.

The cardinalities of the sets $P(\mu)$ and $F(\mu)$ are given by
\begin{align*}
|P(\mu)| = 
\left\{
\begin{array}{ll}
  2^{|\supp{\mu}|-1},&\mu \neq 0,\\
  1, &\mu = 0,
\end{array}
\right.
\textrm{ and } 
|F(\mu)| = \prod_{m=1}^{\length{\mu}} \left \lfloor \frac{\mu_m}{2}+1 \right \rfloor,
\end{align*}
respectively, from which we get the following estimate for the cardinality of $W(\mu)$:
\begin{align}
\label{equ:w}
|W(\mu)| \leq |P(\mu)| |F(\mu)| = 
\left\{
\begin{array}{ll}
  2^{|\supp{\mu}|-1} \prod_{m=1}^{\length{\mu}} \left \lfloor \frac{\mu_m}{2}+1 \right \rfloor,&\mu \neq 0,\\
  1, &\mu = 0.
\end{array}
\right.
\end{align}
From $\eqref{equ:w}$, we obtain a (conservative) upper bound for the number of (sparse) matrix additions required for the computation of the summed matrices in the second step:
\begin{align*}
\sum_{\mu\in\Xi}& |W(\mu)| \leq 1 + \sum_{\substack{\mu\in\Xi\\ \mu\neq 0}} 2^{|\supp{\mu}|-1} \prod_{m=1}^{\length{\mu}} \left \lfloor \frac{\mu_m}{2}+1 \right \rfloor\\
&\leq 1 + \sum_{\substack{\mu\in\Xi\\ \mu\neq 0}} 2^{\maxSV-1} \left(\frac{\max(\mu)}{2}+1 \right)^{\maxSV}
\leq \frac{|\Xi|}{2}\left(\underset{\mu\in\Xi}{\max}(\max(\mu))+2 \right)^{\maxSV}.
\end{align*}
Notice that the above complexity estimate depends entirely on the stochastic input, and hence the time spent on the second step is essentially imposed by the problem formulation; the simpler the stochastic input, the less time we are spending on computing the summed matrices. 
In the simplest case, i.e., in the affine one, 
$W(\mu) = \{\mu\}$, and hence $S^\mu = N^\mu$ and the second step can be skipped.

By combining the complexity estimates for the first and second steps, the overall complexity of locating nonzero elements in the stochastic moment matrices is
\begin{align*}
\mathcal{O}\left(|W_\Xi| |\Lambda|^{1+\epsilon} +  \frac{|\Xi|}{2}\left(\underset{\mu\in\Xi}{\max}(\max(\mu))+2 \right)^{\maxSV}\right).
\end{align*} 
If we assume the stochastic input in the problem formulation, i.e., $\Xi$ and $\maxSV$, to be fixed, the complexity estimate amounts to  
$\mathcal{O}\left(|\Lambda|^{1+\epsilon}\right)$.

Employing the neighbour matrices is profitable if the time used to construct them and the summed matrices is less than the time used to evaluate the zero elements in the stochastic moment matrices.
If the stochastic moment matrices are sparse, the time spent on evaluating the zero elements is essentially $\mathcal{O}(|\Lambda|^2)$ using \eqref{equ:Gdef} directly and $\mathcal{O}(|\Lambda|^{1+\epsilon})$, where $0 < \epsilon \leq 1$, if we employ the neighbour matrices.
However, if the neighbour matrices are dense, the time spent on evaluating the zero elements using \eqref{equ:Gdef} directly is practically zero, whereas the time to construct the neighbour matrices has not changed, and hence it is better to construct the stochastic moment matrices using \eqref{equ:Gdef} directly.
Be that as it may, by employing the neighbour matrices, we have been able to reduce the stochastic moment matrix construction time from several hours to a few minutes in many practical cases.

\section{Constructing hierarchical multi-indices and neighbour matrices}
\label{sec:cm}

In this section, we present algorithms that can be used to construct the $\TS, \TD$, and $\TP$ index sets and the corresponding neighbour matrices $\{N^w\}_{w \in W_\Xi}$.
We give both conceptual (Algorithms~\ref{alg:generalTS} and \ref{alg:generalNfinding}) and pseudocode (Algorithms~\ref{alg:multi1} and \ref{alg:Gmatrices}) implementations of the algorithms.
We consider first the $\TS$ index set and then give the small modifications required for the $\TD$ and $\TP$ index sets.
We conclude by considering other similar index sets.

\subsection{Threshold sequence index set}

The algorithm given in Appendix~\ref{App:A}, i.e., Algorithm~\ref{alg:multi1}, which is the stack-based version of the recursive algorithm presented in \cite{Bieri09b}, can be used to construct the index set $\TS(\mu,\varepsilon)$.
Subsequently, the neighbour matrices $\{N^w\}_{w\in W_\Xi}$ can be constructed using Algorithm~\ref{alg:Gmatrices} from Appendix~\ref{App:B}.

Algorithm~\ref{alg:multi1} takes a sequence $\mu\in c_0$ and a tolerance $\varepsilon>0$ as input and returns the index set $\TS(\mu,\varepsilon)$ together with the corresponding parenthood information (cf. Definition~\ref{def:parenthoodcondition}); the parenthood information is required in the construction of the neighbour matrices in Algorithm~\ref{alg:Gmatrices}.

\subsubsection{High-level description of the $\TS$ index set construction.}

To ease the understanding of Algorithm~\ref{alg:multi1}, a high-level description of the algorithm is given in Algorithm~\ref{alg:generalTS} and explained more in detail below. The line(s) indicated at the end of the lines in Algorithm~\ref{alg:generalTS} correspond to the relevant pseudocode lines in Algorithm~\ref{alg:multi1}.

Let $\beta$ be a child of a multi-index $\alpha$. The multi-index $\beta$ is a \textit{feasible child}
of $\alpha$ if the index set condition $\mu^\beta\geq\varepsilon$ holds.
After  $\alpha$ has been added to the index set under construction, the feasible children 
such as $\beta$
are added to the index set in the decreasing order of length before any other multi-indices are added to the index set.
Moreover, the parenthood information of the just added multi-index is recorded in a natural tree-format.
The index set $\TS(\mu,\varepsilon)$ is always initialized with zero multi-index.

The feasible children are processed in the decreasing order of length to ensure compatibility with the recursive version of the algorithm presented in \cite{Bieri09b} and to facilitate efficient neighbour matrix construction later in this section.
(Any consistent choice here would work with some minor modifications to the algorithms.)

The algorithm is guaranteed to stop at some point as there is a finite number of multi-indices for which the condition $\mu^\alpha \geq \varepsilon$ holds; see~\cite{Bieri09a} for the analysis of the running time and for bounds of the size of the index set $\TS(\mu,\varepsilon)$. The algorithm generates the set $\TS(\mu,\varepsilon)$ essentially due to the fact that if $\mu^\alpha\mu_n < \varepsilon$ then, by recalling that $\mu\in c_0$, it holds that $\mu^\alpha\mu_m < \varepsilon$ for all $m>n$ as well. Hence, if a child of the current multi-index is not feasible, we know that any sequence of one-increments to the child would not produce a multi-index that would belong to the index set, and therefore there is no need to consider it. An example run of Algorithm~\ref{alg:multi1} with intermediate steps in the spirit of Algorithm~\ref{alg:generalTS} is presented in Appendix~\ref{App:examplerun}.

\begin{algorithm}
\caption{High-level description of the $\TS$ index set construction.}
 \label{alg:generalTS}
 \SetLine 
  set current multi-index to $(0,0,\ldots)$ and add it to $\TS(\mu,\varepsilon)$\tcp*[r]{1-5}
  \While{True}{
    push all feasible children of the current multi-index into a stack in the increasing order of length\tcp*[r]{7-12}
      \lIf(\tcp*[f]{13}){stack is empty}{return $\TS(\mu,\varepsilon)$ and the corresponding parenthood information}
      set the current multi-index and the associated variables to the top state in the stack and pop the stack afterwards\tcp*[r]{14-18}
    add the current multi-index to $\TS(\mu,\varepsilon)$ and record the corresponding parenthood information\tcp*[r]{19-23}
  }
\end{algorithm}

\subsubsection{High-level description of the neighbour matrix construction}
Algorithm~\ref{alg:Gmatrices}, given in Appendix~\ref{App:B}, can be used to construct the neighbour matrices $\{N^w\}_{w\in W_\Xi}$.
The algorithm
takes the weights $W_\Xi$, the index set $\TS(\mu,\varepsilon)$, and the corresponding vector $\mu$, real number $\varepsilon$, and parenthood relations
as input, i.e., the input and output of Algorithm~\ref{alg:multi1}, and returns the neighbour matrices $\{N^w\}_{w\in W_\Xi}$.

The high-level description of Algorithm~\ref{alg:Gmatrices} is given in Algorithm~\ref{alg:generalNfinding} in a similar format as Algorithm~\ref{alg:generalTS}
and is essentially: for each $w \in W_\Xi$ loop over each $\eta \in \TS(\mu,\varepsilon)$ and if $\gamma = \eta - w$ belongs to $\TS(\mu,\varepsilon)$, locate $\gamma \in \TS(\mu,\varepsilon)$ and add the corresponding contribution to $N^w$. 
This method requires an efficient way for both 
checking whether $\gamma$ belongs to $\TS(\mu,\varepsilon)$
and locating $\gamma$ in $\TS(\mu,\varepsilon)$. 

A simple and efficient test for $\gamma = \eta - w \notin \TS(\mu,\varepsilon)$ is given by:
\[
\gamma\notin\TS(\mu,\varepsilon) \Leftrightarrow \gamma_m<0 \textrm{ for some } m \textrm{ or } \mu^\gamma = \mu^{\eta - w} = \mu^\eta / \mu^w < \varepsilon.
\]
We need to check $\gamma$ for possible negative values as after subtracting $w$ from $\eta$ we may end up with negative values in $\gamma$. Moreover, notice that we are storing the value $\mu^\eta$ already in the sparse format of $\eta$ in $\TS(\mu,\varepsilon)$, and hence it is enough to compute the value $\mu^w$ once for each $w\in W_\Xi$. 

The efficient location of $\gamma$ in $\TS(\mu,\varepsilon)$ is based on the tree structure provided by the parenthood relations. 
We can locate a given multi-index $\gamma \in \TS(\mu,\varepsilon)$ by starting from the root of the tree, i.e., $(0,0,\ldots)$, after which we move down the tree by processing the position-value pairs $(m,\gamma_m)$ in the sparse format of $\gamma$ in the increasing position order using the procedure:
\begin{enumerate}
\item Move one step down the tree to the child with the length $m$.
\item Move $\gamma_m - 1$ times down the tree selecting the child with the length $m$.
\end{enumerate}
In our implementation, the second step corresponds to selecting the first child in the child list of the current multi-index $m$ times.

For example, the multi-index $(2,0,1,0,\ldots) \equiv \{ (1,2), (3,1) \}$
in the example run of Algorithm~\ref{alg:multi1} presented in Appendix~\ref{App:examplerun} is located through the steps:
\[\lefteqn{\overbrace{\phantom{ (0,0,0) \rightarrow (1,0,0) \rightarrow (2,0,0)}}^{(1,2)}}(0,0,0) \rightarrow (1,0,0) \rightarrow \underbrace{(2,0,0) \rightarrow (2,0,1)}_{(3,1)}.\]
Here, we have used the consistent recording order of the parenthood relations
and the fact that only nonzero indices are stored in the sparse format.

\begin{algorithm}[H]
\caption{High-level description of the neighbour matrix construction.}
 \label{alg:generalNfinding}
 \SetLine 
  \For(\tcp*[f]{1-3}){$w \in W_\Xi$}{
    \For(\tcp*[f]{4-6}){$\eta \in \TS(\mu,\varepsilon)$}{
      \lIf(\tcp*[f]{7-8}){$\eta - w \notin \TS(\mu,\varepsilon)$}{continue}
      locate $\eta - w$ in $\TS(\mu,\varepsilon)$\tcp*[r]{9-22}
      add contribution to $N^w$\tcp*[r]{23}
    }
  }
  return $\{N^w\}_{w \in W_\Xi}$\tcp*[r]{26}  
\end{algorithm}

\subsection{Total degree and tensor product index sets}
\label{subsec:tdtp}
Algorithms~\ref{alg:multi1} and \ref{alg:Gmatrices} can be readily modified for the $\TD$ and $\TP$ index sets;
in Algorithm~\ref{alg:multi1} we change the definition of ``feasible children''  
and in Algorithm~\ref{alg:Gmatrices} the set condition for $\eta - w$.

In principle, the $\TD$ case does not require any modifications as it is possible to reduce the $\TD$ case to the $\TS$ case as explained in Section~\ref{sec:preliminaries}.
However, in the $\isoTD$ case we might want to modify the algorithms slightly to avoid any possible problems arising from the floating point arithmetic;
we perform a global replace of $\TS(\mu,\varepsilon)$ to $\isoTD(N,K)$ and the modifications presented in Appendix~\ref{App:C} to Algorithms~\ref{alg:multi1} and \ref{alg:Gmatrices} to bring the algorithms compatible with the set condition \eqref{equ:TDcond} instead of \eqref{equ:TScond}.

In the $\TP$ case, we need to consider only the $\aTP$ index set as the $\aTP$ index set reduces to the $\isoTP$ index set when $\mu$ is a constant weight vector.
As in the $\isoTD$ case, we perform a global replace of $\TS(\mu,\varepsilon)$ to $\aTP(N,K,\mu)$ and the modifications presented in Appendix~\ref{App:D}.

\subsection{Complexity estimates}
\label{sec:comp}
From the high-level descriptions of Algorithms~\ref{alg:generalTS} and \ref{alg:generalNfinding}, we see that the workload of the algorithms is $\mathcal{O}(|\Lambda|)$ and
$\mathcal{O}(|W_\Xi| |\Lambda| \max_{\alpha\in\Lambda} |\alpha|)$, respectively, where $|\alpha| = \sum_{n \geq 1} \alpha_n$.
The value of $\epsilon$ in the estimate $\mathcal{O}(|W_\Xi| |\Lambda|^{1+\epsilon})$ presented at the end of Section~\ref{sec:smm} can be approximated by equating
\begin{align}
\label{equ:compexp}
\max_{\alpha\in\Lambda} |\alpha| = |\Lambda|^\epsilon \Leftrightarrow \epsilon = \log\left(\max_{\alpha\in\Lambda} |\alpha|\right)/\log(|\Lambda|).
\end{align}
From the estimates
\begin{equation}
\label{equ:compa}
\begin{aligned}
\max_{\alpha\in\aTP(N,K,\mu)} |\alpha|
&\leq
\max_{\alpha\in\isoTP(N,K)} |\alpha|
= NK,\\
\max_{\alpha\in\aTD(N,K,\mu)} |\alpha|
&\leq
\max_{\alpha\in\isoTD(N,K)} |\alpha|
= K,
\end{aligned}
\end{equation}
and
\begin{equation}
\label{equ:compb}
\begin{aligned}
\left(\left\lfloor \frac{\mu_{\min}}{\mu_{\max}} K  \right\rfloor + 1 \right)^N
\leq
|\aTP(N,K,\mu)|
&\leq
|\isoTP(N,K)|
= (K+1)^N,
\\
{{N + \left\lfloor \frac{\mu_{\min}}{\mu_{\max}} \right\rfloor K} \choose N}
\leq
|\aTD(N,K,\mu)|
&\leq
|\isoTD(N,K)|
= {{N+K} \choose N},
\end{aligned}
\end{equation}
we get by combining \eqref{equ:compa} and \eqref{equ:compb} in \eqref{equ:compexp} and letting $K\rightarrow\infty$ the upper bound
\begin{align*}
\lim_{K\to\infty} \epsilon
\leq \frac{1}{N}
\end{align*}
for the $\TP$ and $\TD$ index sets.
In most cases $N = \maxSV$, and hence the complexity of the neighbour matrix construction is approximately $\mathcal{O}\left(|\Lambda|^{1+\maxSV^{-1}}\right)$ for the $\TP$ and $\TD$ index sets.

Timing results for a diffusion coefficient of the form
\begin{align}
\label{equ:timinga}
a(\y,\x) = 1 + \left[1 + \sum_{m=1}^M a_m(\x) y_m\right]^p
\end{align}
and an $\isoTD(M,K)$ index set with different values of $M, p$, and $K$ are presented in Figure~\ref{fig:timing}.
The tables on the left of Figure~\ref{fig:timing} portray, from top to bottom, the estimated value of $\epsilon$ in the complexity estimate $\mathcal{O}(|W_\Xi| |\Lambda|^{1+\epsilon})$ together with the theoretical value $M^{-1}$,
the number of (sparse) matrix additions required for the computation of the summed matrices, i.e.,
$\sum_{\mu\in\Xi} |W(\mu)|$, and the maximum value of $K$ considered when estimating $\epsilon$ for different $M$ and $p$.
The image on the right of Figure~\ref{fig:timing} shows the time required to compute the neighbour matrices for $M=6$ and $p=2,4,6$ together with the line used to estimate $\epsilon$.
The plot markers correspond to different values of $K$, and
the timings were carried out with Intel$^\circledR$ Xeon$^\circledR$ Processor E3-1230 V2 (8M Cache, 3.30GHz) with 15.6 GiB of memory.

The obtained values of $\epsilon$ depicted in Figure~\ref{fig:timing} are in agreement with the theoretical limits. 
Increasing $p$ increases the $y$-intercepts of the fitted lines in the timing plot which is in agreement with the complexity estimate as increasing $p$ amounts to a larger $W_\Xi$.

\begin{figure}[h]
    \centering
\begin{subfigure}{0.35\textwidth}
    \centering
    \vspace{0pt}
    \rowcolors{1}{}{lightgray}
     \begin{tabular}{c|ccc|c}
       & $p=2$ & $p=4$ & $p=6$  & $1/M$\\
      \hline
      $M=2$ & .66  & .67  & .67&.50\\ 
      $M=4$ & .23  & .22  & .24&.25\\
      $M=6$ & .14  & .17  & .20&.17
    \end{tabular}
    \vspace{0.5cm}\\
    \rowcolors{1}{}{lightgray}
     \begin{tabular}{c|ccc}
& $p=2$ & $p=4$ & $p=6$ \\
      \hline
      $M=2$ & 9  & 38 & 110  \\
      $M=4$ & 25 & 255 & 1519\\
      $M=6$ & 49 & 924 & 9324\\      
    \end{tabular}
    \vspace{0.5cm}\\
    \rowcolors{1}{}{lightgray}
     \begin{tabular}{c|ccc}
       & $M=2$ & $M=4$ & $M=6$ \\
      \hline
      $K$& 800  & 60 & 20  \\
    \end{tabular}
\label{fig:horse}
\end{subfigure}
\hfill
\begin{subfigure}{0.6\textwidth}
\centering
\includegraphics{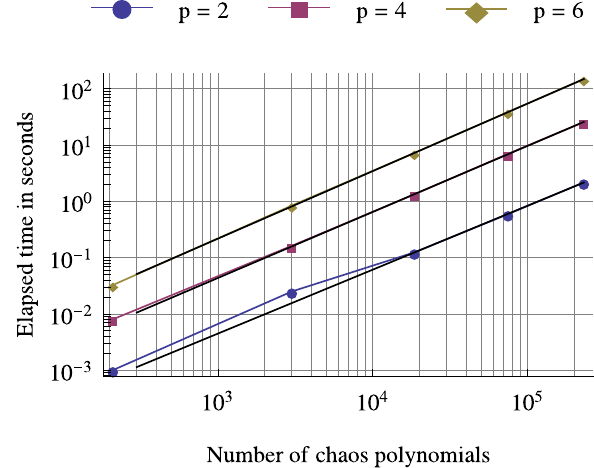}
\end{subfigure}
\caption{
Timing results for a diffusion coefficient of the form~\eqref{equ:timinga}
and an $\isoTD(M,K)$ index set with different values of $M, p$, and $K$.
Left: 
From top to bottom,
estimated value of $\epsilon$ in the complexity estimate $\mathcal{O}(|W_\Xi| |\Lambda|^{1+\epsilon})$ together with the theoretical value $M^{-1}$,
the number of (sparse) matrix additions required for the computation of the summed matrices, i.e.,
$\sum_{\mu\in\Xi} |W(\mu)|$, and the maximum value of $K$ considered when estimating $\epsilon$ for different $M$ and $p$.
Right: 
Time required to compute the neighbour matrices for $M=6$ and $p=2,4,6$ together with the line used to estimate $\epsilon$.
The plot markers correspond to different values of $K$.
}
\label{fig:timing}
\end{figure}

\subsection{Other index set types}

It is possible to consider other index sets, say $\Lambda$, than the $\TS$, $\TD$, and $\TP$ index sets using Algorithms~\ref{alg:multi1} and \ref{alg:Gmatrices} with (minor) modifications.
What is required is a method to check whether a given multi-index is an element of $\Lambda$ and some structure like 
the parenthood relations 
to locate it.
Currently, the tree structure provided by the parenthood relations requires that the index set $\Lambda$ can be constructed by adding an initial multi-index to $\Lambda$, after which all additions to $\Lambda$ are obtained by applying a one-increment to the last nonzero or higher index in a multi-index already in $\Lambda$.

An example of a relevant index set for which it is not straightforward to modify Algorithms~\ref{alg:multi1} and \ref{alg:Gmatrices} is the {\em total degree with factorial correction} $\TDFC$ index set considered by Beck et al.~\cite{Beck11}:
\begin{definition}[Total degree with factorial correction index set]
Let $N \in \N$,
$\mu\in\R^N$,
and $\varepsilon > 0$. 
The $\TDFC$ index set is given by
\begin{align}
\label{equ:TDFCcond}
\TDFC(N,\mu,\varepsilon) = \left\{\alpha\in (\mathbb{N}_{0}^\infty)_c~|~\sum_{n=1}^{N}\mu_n\alpha_n - \log{\frac{|\alpha|!}{\alpha!}} \leq \lceil-\log{\varepsilon} \rceil, \alpha_n = 0, n>N \right\},
\end{align}
where $\alpha! = \prod_{n \geq 1} \alpha_n!$.
\end{definition}

We cannot use the presented algorithms
directly  for the $\TDFC$ index set as it is possible that if $\alpha \notin \TDFC(N,\mu,\varepsilon)$ then $\alpha + \beta \in \TDFC(N,\mu,\varepsilon)$, where $\beta \in (\mathbb{N}_{0}^\infty)_c$; for example, the multi-index $(2,0,0,\ldots)$ does not belong to the index set $\TDFC(2,(1,1),e^{-1.95})$ but $(2,1,0,0,\ldots)$ does.

To conclude this section, let us make a short comment about the general case. Assume we are given an index set, say $\Lambda$, with no topological information, i.e., the actual multi-indices in $\Lambda$ have been generated, and we are asked to construct the neighbour matrices. 
Clearly Algorithm~\ref{alg:Gmatrices} (or some variant of it) cannot be used directly as we do not have the tree structure provided by the parenthood relations available (or some other similar data structure).
One option would be to construct a suitable tree structure, e.g., a KD-tree, and then use a modified version of Algorithm~\ref{alg:Gmatrices} to construct the neighbour matrices.
Constructing a KD-tree is a loglinear operation, and hence the complexity estimate for the construction of the neighbour matrices would not change in big O notation as $\mathcal{O}(|\Lambda| \log |\Lambda| + |\Lambda|^{1+\epsilon}) = \mathcal{O}(|\Lambda|^{1+\epsilon})$.
However, in practice it could be more efficient to first construct a hash table for $\Lambda$ using a suitable hash function and then use the hash table in Algorithm~\ref{alg:generalNfinding} to search elements in $\Lambda$.
Employing hash functions in the construction of the stochastic moment matrices is in itself interesting and clearly out of the scope of this work, and hence we leave these type of considerations for future work.

\section{Numerical experiments}
\label{sec:ne}

We consider solving the following one-dimensional elliptic model problem in the Legendre case:
find random field $u \in L_{\rho}^2(\Gamma, H_0^1(D))$ such that
\begin{align}
\label{equ:1Dmodelproblem}
\left\{
  \begin{array}{ll}
    -(a(\y,x) u(\y,x)')' = 1 &\textrm{ in } D=(-1,1),\\
    u(\y,-1) = u(\y,1) = 0
  \end{array}
\right.
\end{align}
holds for a given diffusion coefficient $a(\y,x)$ for which $a_{\textrm{min}}$ and $a_{\textrm{max}}$ as in condition \eqref{equ:assumptionona} exist. 

We consider 
the space-independent diffusion coefficient
\begin{align*}
a(\y,x) = a(\y) = 1 + \left[1 + \sum_{m=1}^M m^{-s}y_m\right]^p
\end{align*}
and
the space-dependent diffusion coefficient
\begin{align}
\label{equ:spacedepform}
a(\y,x) = 1 + \left[1 + \sum_{m=1}^M m^{-s} \sin(m \pi x)y_m\right]^p
\end{align}
with different values of $M$, $s$, and $p$ as summarized in Table~\ref{tab:params}.
We have fixed the overall form of the diffusion coefficient and only varied the number of used random variables and the power in the diffusion coefficient; we could have considered other types of diffusion coefficients, with similar conclusions.
For the spatial discretization, we employ 20 $p$-type elements of order four.

\begin{table}[b]
\rowcolors{1}{}{lightgray}
\begin{tabular}[]{rlccc}
&Diffusion coefficient type&$M$&$s$&$p$\\
\hline
Example 1 & space-independent & 2 & 3/2 & 6\\
Example 2 & space-dependent &4 & 3/2 & 2\\
Example 3 & space-dependent &6 & 3/2 & 4\\
\end{tabular}
\caption{
  Summary of the diffusion coefficient type and parameters $M$, $s$, and $p$ in Examples~1--3.
}
\label{tab:params}
\end{table}

In the first example, we use a space-independent diffusion coefficient for which we have the exact solution available,
and hence we can compare the performance of the various index sets against the exact result.
The second and third examples are space-dependent for which we do not have a simple analytical solution available, and thus we rely on a case-specific ``accurate enough'' sGFEM solution as the ``exact'' solution.

Even though the first example has a higher power in the diffusion coefficient than the two other examples,
the small number of random variables and space-independency make it considerably easier to solve than the second or third example.
Similarly, the second example is easier to solve than the third example as it has fewer random variables and a smaller power in the diffusion coefficient, which results in more sparse system matrices.
In summary, we expect to require larger index sets, i.e., more polynomials from the polynomial chaos, in the second and third examples than in the first one.

\subsection{Design of the numerical experiments}

We denote $H_0^1(D)$ with $V$ and equip it with the gradient norm $\|v\|_V = \| \nabla v \|_{L^2(D)}, \forall v\in V$, which induces the same topology as the standard $H$-norm due to the homogeneous Dirichlet condition.
The solution to \eqref{equ:1Dmodelproblem} can be represented in terms of the Legendre polynomials as~\cite{Schwab11a}
\begin{align}
\label{equ:uexpansion}
u(\y, x) = \sum_{\mu\in(\mathbb{N}_{0}^\infty)_c} u_{\mu}( x)L_{\mu}(\y)
\end{align}
in the topology of $L_\rho^2(\Gamma)$,
where the Legendre ``coefficient'' $u_{\mu} \in H_0^1(D)$ is given by
\[
u_{\mu}( x) = \int_\Gamma u(\y,  x) L_\mu(\y) \rho(\y) \dd \y.
\]
By considering only a finite number of terms from \eqref{equ:uexpansion}, we obtain an approximation for $u$
\[
u(\y,  x) \approx u_\Lambda(\y, x) = \sum_{\mu\in \Lambda} u_{\mu}( x)L_{\mu}(\y),
\]
where $\Lambda$ is a finite multi-index set.
The projection error is given by
\[
\| u - u_\Lambda \|_{L_\rho^2(\Gamma,V)}^2 = \sum_{\mu \notin \Lambda} \| u_\mu \|_V^2
\]
which follows from Parseval's equality and the completeness of $\{ L_\mu \}_{\mu \in (\mathbb{N}_{0}^\infty)_c}$ in $L_\rho^2(\Gamma)$.

Based on~\cite{Cohen11}, it is assumed
that for our numerical examples the following inequality for the Legendre coefficient norms holds for 
$s > 1$ and any $\epsilon > 0$ when $M$ approaches infinity:
\begin{align}
\label{equ:convrate}
\sqrt{\sum_{\mu \notin \Lambda} \| u_\mu \|_V^2} \leq C |\Lambda|^{1-s+\epsilon} \approx C |\Lambda|^{1-s},
\end{align}
where $C$ is a suitable constant.
Inequality~\eqref{equ:convrate} implies the (weaker) inequality for $\mu\notin\Lambda$
\begin{align}
\label{equ:convrate2}
\| u_\mu \|_V \leq C  |\Lambda|^{1-s+\epsilon}\approx C |\Lambda|^{1-s}.
\end{align}
We compare numerically observed Legendre coefficient convergence rates to~\eqref{equ:convrate2} in Section~\ref{sec:convergencerates}.

By solving \eqref{equ:reslinearsystem2} using a finite multi-index set $\Lambda$ (and a fine enough FEM grid), we obtain an sGFEM approximation for $u$
\[
u(\y,  x) \approx \tilde{u}_\Lambda(\y, x) = \sum_{\mu\in \Lambda} \tilde{u}_{\mu}( x)L_{\mu}(\y),
\]
where 
$\{ \tilde{u}_\mu \}_{\mu\in\Lambda}$ are approximations for the Legendre coefficients $\{ u_\mu \}_{\mu\in\Lambda}$.
We obtain the {\em best $M$-terms} approximation for $u$ by first computing the (approximative) Legendre coefficients of $u$ in a sufficiently large index set $\Lambda$, i.e., $\{ \tilde{u}_\mu \}_{\mu \in \Lambda}$, and then ordering the coefficients in a non-increasing order based on their $V$-norm and taking the first $M$-terms from the resulting ordered sequence. The corresponding $M$ multi-indices form the $\bestM$ 
index set.

We report the performance of 
the $\TP$, $\TD$, and $\bestM$ 
index sets 
by considering the convergence of the relative errors of mean and variance (percentage), i.e, 
\begin{align*}
\frac{\big\|\mathbb{E}(\tilde{u}_\Lambda) - \mathbb{E}(u)\big\|_{L^2(D)}}{\big\|\mathbb{E}(u)\big\|_{L^2(D)}}\times100 \quad \textrm{ and } \quad
\frac{\big\|\operatorname{Var}(\tilde{u}_\Lambda) - \operatorname{Var}(u)\big\|_{L^2(D)}}{\big\|\operatorname{Var}(u)\big\|_{L^2(D)}}\times100,
\end{align*}
when the number of multi-indices in the corresponding index set is increased.
Here, ``variance'' is to be understood as the diagonal of the corresponding covariance function and
the $L^2(D)$-norms are evaluated using built in numerical integration routines of Mathematica~\cite{Mathematica2012}.
We denote the case-specific sGFEM solution functioning as the ``reference'' exact solution with 
$u$ even though $\tilde{u}_\Lambda$ with a suitable index set $\Lambda$ would be more appropriate.

The weights for the dimensions used in the anisotropic index sets are computed numerically using 1D analyses as in \cite{Beck11,Beck12}.
For each random variable $1\leq m \leq M$, we consider the subset $\Lambda_m = \{ \mu \in \Lambda : \mu_i = 0 \textrm{ if } i \neq m, \mu_m = 0,1,2,\ldots \}$, where $\Lambda \subset (\mathbb{N}_{0}^\infty)_c$ is a sufficiently large index set.
We assume the decay of the Legendre coefficients corresponding to $\Lambda_m$ to be of the form $ \|u_\mu \|_V \sim\exp(-g_m \mu_m)$, 
and we obtain the rate $g_m$ by linear interpolation on the quantities $\log \|u_\mu \|_V$.

The numerical results for the experiments are obtained in the following way.
First, the solutions corresponding to the $\isoTP(M,K)$ and $\isoTD(M,K)$ index sets, i.e., the $\isoTP(M,K)$ and $\isoTD(M,K)$ solutions, are computed 
by increasing $K$ from zero to some suitable maximum value.
Then the weight vector $g$ is approximated from one of the computed sGFEM solutions as explained above.
Once the weights are available, the 
anisotropic solutions are computed up to a suitably large index set.
If the exact solution is not available, one of the computed sGFEM solutions is selected as the reference solution, 
and the relative errors in mean and variance are computed either against the exact or the reference solution.
Finally, the best $M$-terms sGFEM solutions are computed as explained above by using the sGFEM solution corresponding to the largest used index set of the best performing index set type.

In addition to the convergence plots of the relative errors in mean and variance,
it turns out useful to visualize the norms of the Legendre coefficients of an sGFEM solution in descending order.
From this kind of graph, it is possible to see how much excess multi-indices, i.e, multi-indices whose corresponding Legendre coefficients have small norms, are included in the index set. 
The ``better'' the index set type, the less excess multi-indices are included in the index set when the size of the index set is increased, i.e., the ``tail'' of the index set stays small.

In the first example, where we consider only two random variables, we also illustrate the largest considered index sets to give better understanding of the different index set types. Moreover, by plotting the norms of the Legendre coefficients as a function of the corresponding multi-indices, we are able to visualize why certain index sets are performing better than the rest of the index sets in that specific case.

\subsection{Space-independent case: Example 1}
\label{sec:n1}
As the first example, we consider a space-independent diffusion coefficient with two random variables
\begin{align}
\label{equ:ex1a}
a(\y, x) = a(\y) = 1 + \left[1 + \sum_{m=1}^2 m^{-3/2}y_m\right]^6.
\end{align}
In this case, the exact solution to \eqref{equ:1Dmodelproblem} is given by 
\begin{align*}
u(\y, x) = \frac{1- x^2}{2a(\y)}
\end{align*}
and the main results are shown in Figures~\ref{fig:example1} and \ref{fig:example1_legendrecoefs}.

Figure~\ref{fig:example1} illustrates the convergence of the sGFEM solutions
toward the exact mean (left) and variance (right) for the first experiment as the number of chaos polynomials, i.e., the number of multi-indices in the index set, is increased. 
From the figure, we see that the best results are obtained, as expected, with the $\bestM$ index set followed by the $\aTP$, $\aTD$, $\isoTD$, and $\isoTP$ index sets.
The anisotropic index sets 
are constructed using the weight vector
$g \approx (1.16, 3.82)$, which is approximated from the $\isoTD(2,19)$ solution,
and the $\bestM$ index set is constructed based on the largest $\aTP$ solution.

From Figure~\ref{fig:example1}, we see that the relative error in mean (variance) when using approximately 200 chaos polynomials is roughly $10^{-4}\%\,(10^{-3}\%)$ in the $\isoTP$ case, $10^{-7}\%\,(10^{-6}\%)$ in the $\aTP$ case, and $10^{-8}\%\,(10^{-7}\%)$ in the best $M$-terms case, and hence using different index set types may give dramatically better results.
The $\aTP$ index set solutions follow first quite nicely the $\bestM$ index set solutions until the relative error between the index sets is about one decade at approximately 200 multi-indices, indicating that the $\aTP$ index set is close to being optimal at least when the size of the index set is not too large. Recall, however, that the $\bestM$ index set is constructed based on the $\aTP$ solution with the largest index set.

The solutions corresponding to the other index sets than to the $\isoTD$ index set converge relatively smoothly both in mean and variance whereas 
the solution corresponding to the $\isoTD$ index set is oscillating when considering the variance and stalling at every other increment when considering the mean.
This reminds us that 
employing more computational resources to increase the size of the index set only slightly does not necessarily give better results as the solution may be oscillating even though it converges asymptotically.

\begin{figure}
  \includegraphics{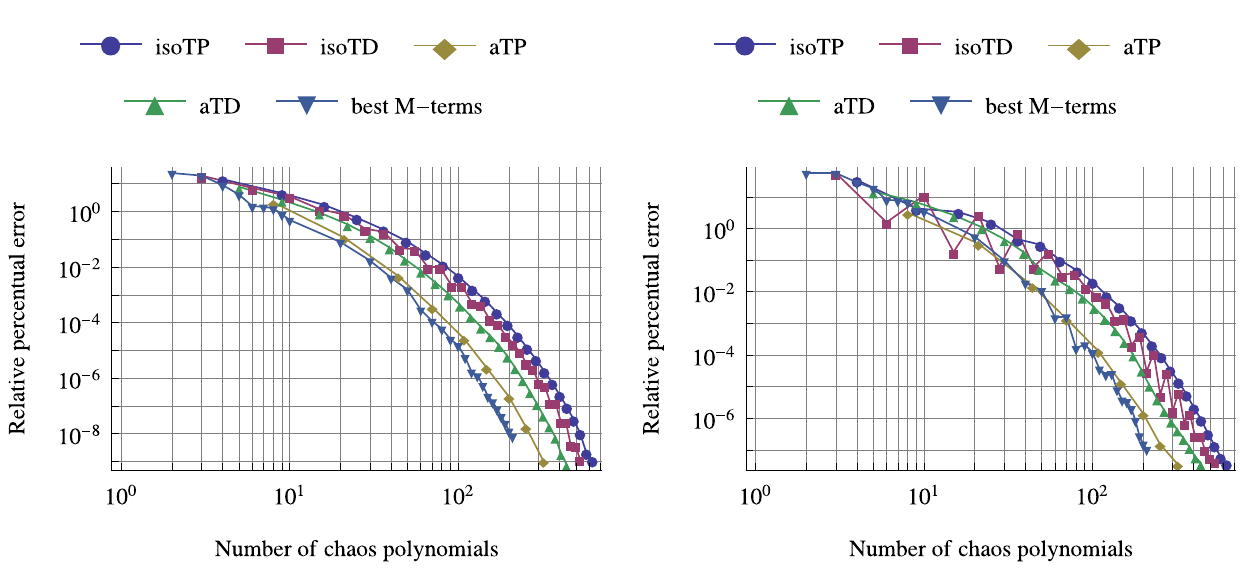}
  \caption{
    Convergence of the $\TP$, $\TD$, and $\bestM$ index set solutions toward the exact mean (left) and variance (right) for the first experiment as the dimension of the polynomial space increases.
    The diffusion coefficient is given by \eqref{equ:ex1a}
    and 
    the anisotropic index sets 
    are constructed using the weight vector
    $g \approx (1.16, 3.82)$ which is approximated from the $\isoTD(2,19)$ solution.
   }
  \label{fig:example1}
\end{figure}

Figure~\ref{fig:example1_indexsets} shows the logarithm of the Legendre coefficient norms for the largest used $\isoTP$ index set, i.e., $\isoTP(2,24)$, on the left and the largest used index sets on the right for the first example.
Notice that the range in the left image is different from the images on the right and that the left image and the top left image on the right correspond to each other; the latter shows the index set and the former portrays the log-norms of the corresponding Legendre coefficients.
In a two-dimensional case,
the isotropic/anisotropic tensor product and total degree index sets can be visualized as squares/rectangles and isosceles right triangles/right triangles, respectively.
By looking at the log-norms shown in Figure~\ref{fig:example1_indexsets}, we see that 
expanding the index set in the direction of the horizontal axis is more profitable than in the direction of the vertical axis.
This agrees with the results in Figure~\ref{fig:example1}:
the 
anisotropic
index sets are performing better than the 
isotropic
index sets.

\begin{figure}
  \includegraphics{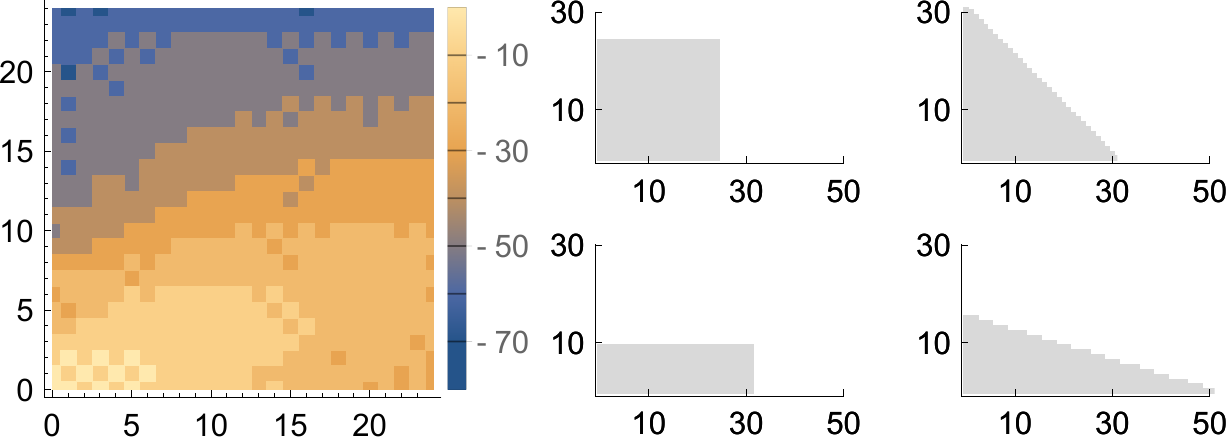}
  \caption{
    The logarithm of the Legendre coefficient norms for the $\isoTP(2,24)$ solution for the first example (left) and the largest used $\TP$ and $\TD$ index sets (right).
    The multi-indices in the index sets of the first example can be given as $\alpha = (\alpha_1,\alpha_2,0,0,\ldots)$, where $\alpha_1, \alpha_2 \in \N_0$.
    Here, the horizontal axis corresponds to $\alpha_1$ and the vertical axis corresponds to $\alpha_2$.
    Notice that the range in the left image is different from those in the images on the right and that the left image and the top left image on the right correspond to each other.
   }
  \label{fig:example1_indexsets}
\end{figure}

Figure~\ref{fig:example1_legendrecoefs} illustrates 
the evolution of the ordered Legendre coefficient norms for the index sets considered in the first experiment as the dimension of the polynomial space increases. 
It seems that the best performing index sets have the smallest tails and 
that the tails grow slower for the anisotropic index sets than for the isotropic index sets, giving a new confirmation for the results in Figure~\ref{fig:example1}. Notice that the largest index sets have almost equal relative errors as depicted by Figure~\ref{fig:example1}, and hence comparing the largest index sets 
is fair.
If we want to decrease the size of an index set for which we have already computed the sGFEM solution, the ordered log-norm plot can be used to select a suitable threshold level for the multi-indices in the smaller index set. 
Finally, we have included an approximated convergence line with rate $r = 2.27$ in the log--log plots of Figure~\ref{fig:example1_legendrecoefs} to allow a comparison between the Legendre coefficients in the different examples;
the larger the rate, the better approximations we are able to obtain with small index sets.
The rate is obtained from a least-squares fit to the 100 largest Legendre coefficients in the $\aTP$ case with the largest index set and it is computed in a similar way for the subsequent examples.
Moreover, see Section~\ref{sec:convergencerates} for a comparison of the observed Legendre coefficient convergence rates to the theoretical results.

\begin{figure}
  \includegraphics{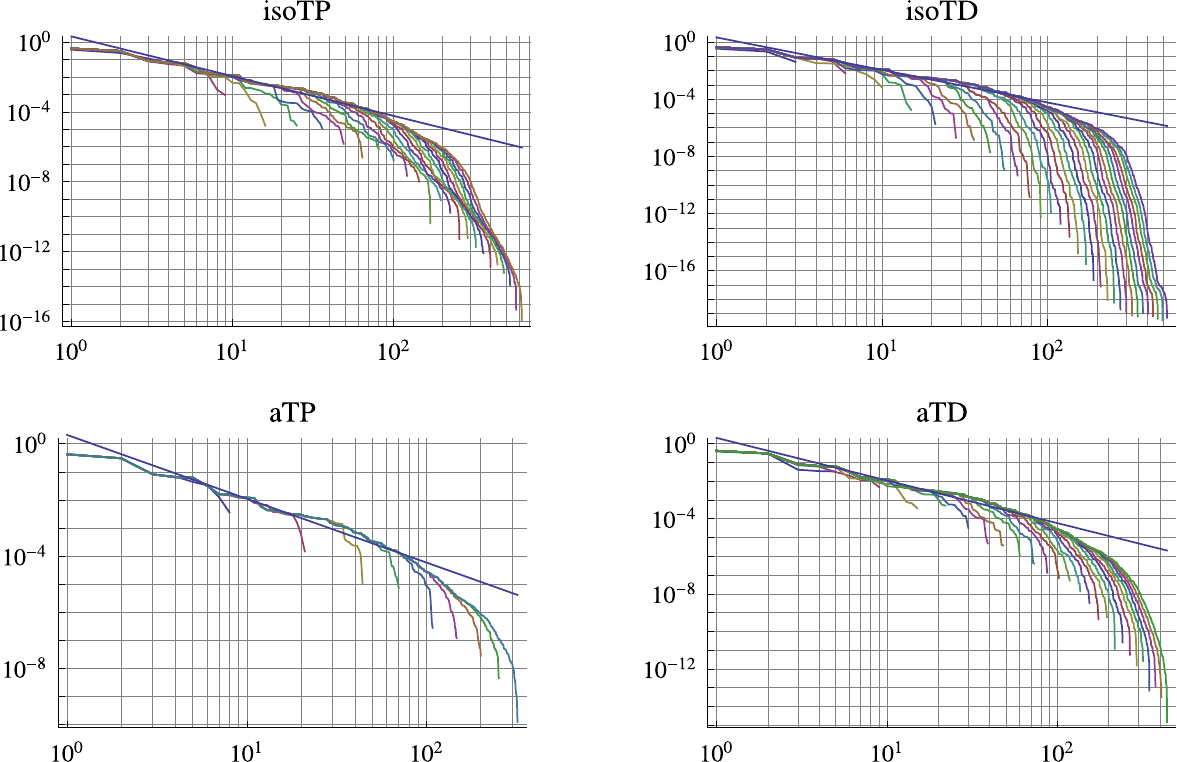}
  \caption{ 
    Evolution of the ordered Legendre coefficient norms (vertical axis) for the index sets considered in the first experiment as the dimension of the polynomial space increases (horizontal axis).
    We have also included an approximated convergence line with rate $r=2.27$ in the images.
   }
  \label{fig:example1_legendrecoefs}
\end{figure}

\subsection{Space-dependent case: Example 2}
\label{sec:n2}
In the second example, we use a space-dependent diffusion coefficient with four random variables
\begin{align}
\label{equ:ex2a}
a(\y, x) = 1 + \left[1 + \sum_{m=1}^4 m^{-3/2}\sin(m\pi x)y_m\right]^2,
\end{align}
and we consider the sGFEM solution corresponding to the $\aTD$ index set with 4078 chaos polynomials as the reference or ``exact'' solution as no simple form for the exact solution is available.

Figures~\ref{fig:example2} and \ref{fig:example2_legendrecoefs}, which are organized in the same way as Figures~\ref{fig:example1} and \ref{fig:example1_legendrecoefs}, respectively, for the first experiment, illustrate the convergence of the mean and variance of the sGFEM solutions toward the reference solution and the evolution of the ordered Legendre coefficients for the different types of index sets for the second example.
Here, the weight vector $g \approx (2.40, 4.17, 5.37, 6.38)$ is approximated from the $\isoTD(4,17)$ solution, and the best $M$-terms index set is constructed based on the largest $\aTD$ solution.

From Figure~\ref{fig:example2}, we see that the relative error in mean (variance) when using approximately $2 \cdot 10^3$ chaos polynomials is roughly 
$10^{-6}\%\, (10^{-4}\%)$ in the $\isoTP$ case, 
$10^{-9}\%\, (10^{-7}\%)$ in the $\isoTD$ and $\aTP$ cases, 
$10^{-12}\%\, (10^{-10}\%)$ in the $\aTD$ case,
and $10^{-13}\%\, (10^{-11}\%)$ in the 
best $M$-terms case.
As in the first example, the different index sets show dramatic differences in performance.
Even though the $\aTP$ index set gave the best performance in the first example,
the performance of the $\aTP$ index set is only on par with the generic $\isoTD$ index set in the second example.
Finally, the $\isoTP$ index set performs poorly when compared with the other index sets.

\begin{figure}
  \includegraphics{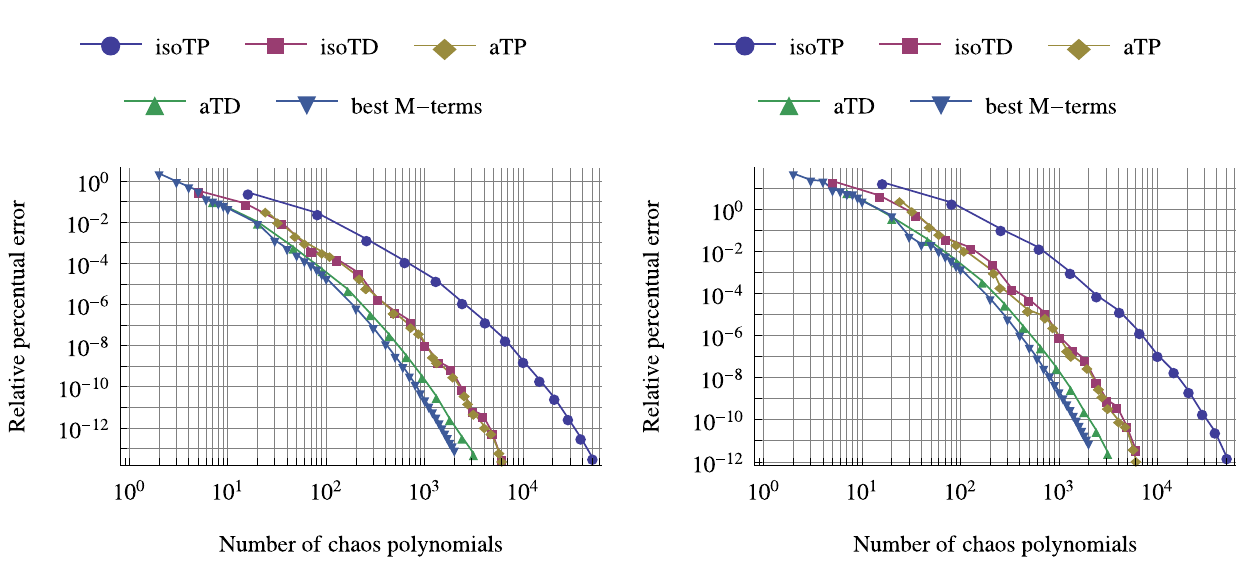}
  \caption{
    Convergence of the $\TP$, $\TD$, and $\bestM$ index set solutions toward the mean (left) and variance (right) corresponding to the $\aTD$ solution with 4078 chaos polynomials for the second experiment as the dimension of the corresponding polynomial space increases.
    The diffusion coefficient is given by \eqref{equ:ex2a}
    and
    the anisotropic index sets
are constructed using the weight vector
    $g \approx (2.40, 4.17, 5.37, 6.38)$ which is approximated from the $\isoTD(4,17)$ solution.
  }
  \label{fig:example2}
\end{figure}

As in the first example, the best performing index sets have the slowest growing tails as illustrated by Figure~\ref{fig:example2_legendrecoefs}.
This suggests the following heuristic method for selecting the index set to use in a practical application. 
First, use different index set types and compute the sGFEM solutions for these index sets using relatively small number of multi-indices.
Finally, choose the index set which looks the most promising based on the evolution of the ordered Legendre coefficient norms.
However, this heuristic method is only feasible if the required sGFEM solutions can be computed in a reasonable time.

The obtained convergence rate for the second example, $r=2.03$, agrees with the convergence results of Figures~\ref{fig:example1} and \ref{fig:example2};
as the rate in the second example is lower than in the first example, we expect that larger index sets are required for a satisfactory converge of the solutions in the second example.
See Section~\ref{sec:convergencerates} for the comparison of the observed Legendre coefficient convergence rates to the theoretical results.

\begin{figure}[h]
  \includegraphics{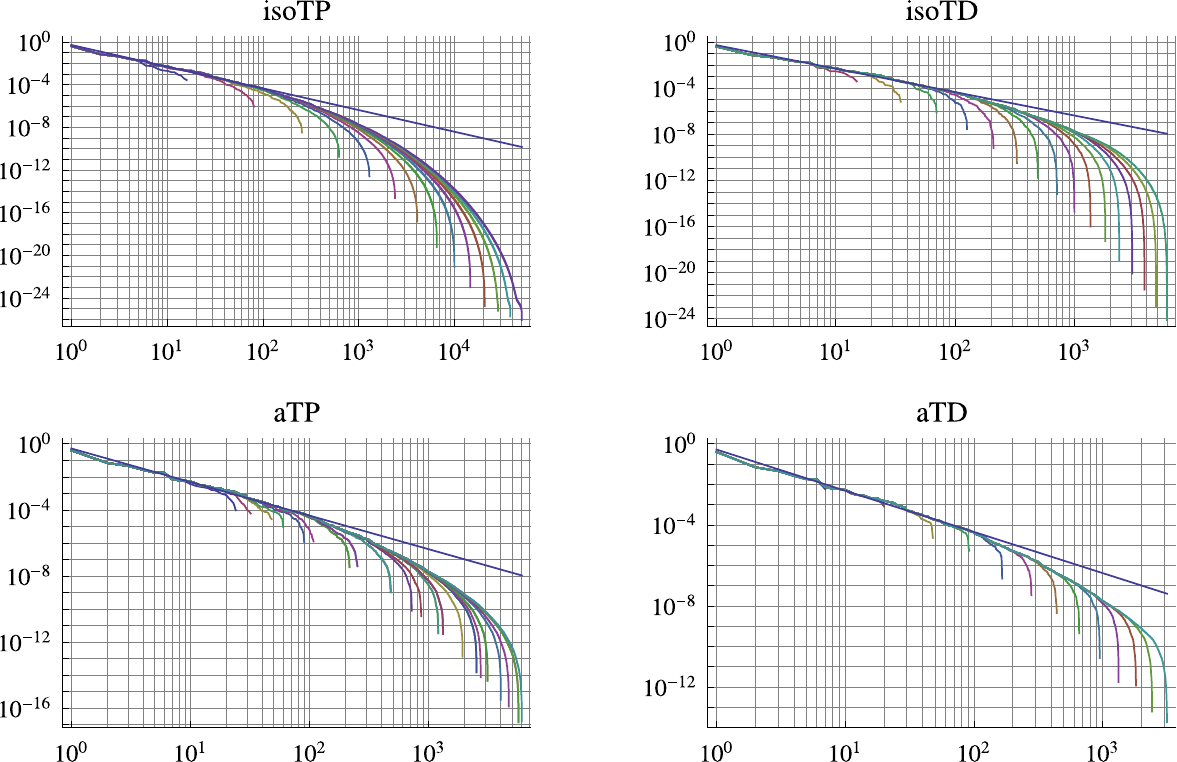}
  \caption{ 
    Evolution of the ordered Legendre coefficient norms (vertical axis) for the index sets considered in the second experiment as the dimension of the polynomial space increases (horizontal axis).
    We have also included an approximated convergence line with rate $r=2.03$ in the images.
   }
  \label{fig:example2_legendrecoefs}
\end{figure}

\subsection{Space-dependent case: Example 3}
\label{sec:n3}
As the third and final example, we use a space-dependent diffusion coefficient with six random variables
\begin{align}
\label{equ:ex3a}
a(\y, x) = 1 + \left[1 + \sum_{m=1}^6 m^{-3/2}\sin(m\pi x)y_m\right]^4
\end{align}
and we consider the sGFEM solution corresponding to the $\aTD$ index set with 5909 chaos polynomials as the reference or ``exact'' solution as no simple form for the exact solution is available.
The weight vector $g \approx (1.52, 3.56, 4.95, 5.76, 6.4, 6.85)$, used in the anisotropic solutions, is approximated from the $\isoTD(6,9)$ solution.

Figure~\ref{fig:example3} depicts the convergence of the relative errors for mean and variance in a similar way as in the previous two examples.
It seems that the $\aTD$ index set clearly outperforms the $\aTP$, $\isoTD$, and $\isoTP$ index sets,
the $\aTP$ index set exhibits slightly better performance than the $\isoTD$ index set, 
and the $\isoTP$ index set, as in the previous two examples, gives the worst performance.
In this example, we were limited by our computing resources to around $5 \cdot 10^3$ multi-indices due to more full system matrices caused by the higher power and number of random variables in the diffusion coefficient than in the second example.
Hence, the largest index sets do not correspond to similar errors but the cardinalities of the largest index sets are about the same.

\begin{figure}
  \includegraphics{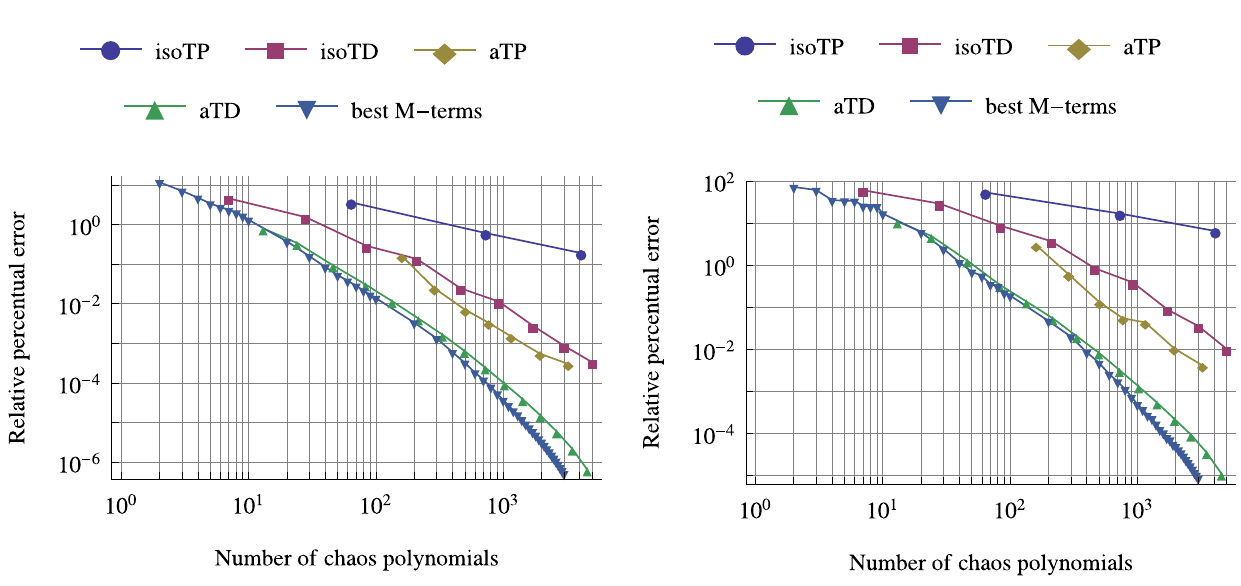}
  \caption{
    Convergence of the $\TP$, $\TD$, and $\bestM$ index set solutions toward the mean (left) and variance (right) corresponding to the $\aTD$ solution with 5909 chaos polynomials for the third experiment as the dimension of the corresponding polynomial space increases.
    The diffusion coefficient is given by \eqref{equ:ex3a}
    and
    the anisotropic index sets are constructed using the weight vector
    $g \approx (1.52, 3.56, 4.95, 5.76, 6.4, 6.85)$ which is approximated from the $\isoTD(6,9)$ solution.
  }
  \label{fig:example3}
\end{figure}

Figure~\ref{fig:example3_legendrecoefs}, which shows the ordered Legendre coefficient norms for the third example, agrees with the previous two examples: the slower the tail is growing for an index set type, the better it is performing.
As we have increased both the number of random variables and the power in the diffusion coefficient, 
it is natural that the approximated convergence line in the log--log plots of Figure~\ref{fig:example3_legendrecoefs} has a smaller rate, $r=1.43$, than in the previous two examples: more chaos polynomials are required to take into account the increased number of random variables and cross terms in the diffusion coefficient. 
According to Figure~\ref{fig:example3_legendrecoefs}, 
the $\isoTP$ index set and the $\isoTD$ and $\aTP$ index sets start to perform worse than the $\aTD$ index set 
at around $10^2$ and $5\cdot 10^2$ multi-indices, respectively, which is in line with the results in Figure~\ref{fig:example3}.

\begin{figure}
  \includegraphics{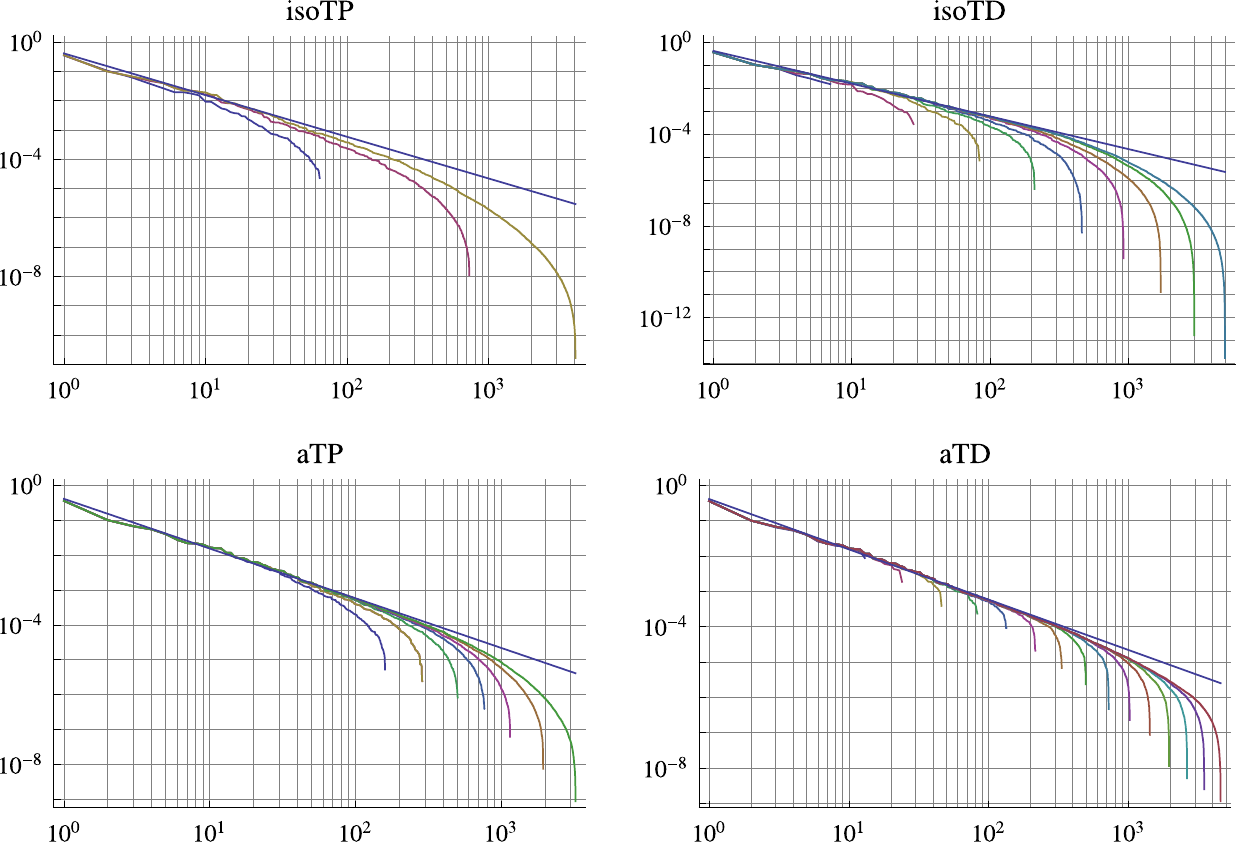}
  \caption{
    Evolution of the ordered Legendre coefficient norms (vertical axis) for the index sets considered in the third experiment as the dimension of the polynomial space increases (horizontal axis).
    We have also included an approximated convergence line with rate $r=1.43$ in the images.
   }
  \label{fig:example3_legendrecoefs}
\end{figure}

\subsection{Legendre coefficient convergence rates}
\label{sec:convergencerates}

To conclude the numerical experiments, we compare the observed Legendre coefficient convergence rates to theoretical results, i.e., 
we examine how close to the theoretical rate $r_s=s-1$ from~\eqref{equ:convrate2} the Legendre coefficients converge when $M$ is large enough;
recall that in our previous three experiments $s=3/2$.
We have listed approximated convergence rates for space-dependent diffusion coefficients of the form~\eqref{equ:spacedepform} with different values of $M$, $s$, and $p$ in Table~\ref{tab:gradetable}.
Considerably larger values of $M$ than shown on Table~\ref{tab:gradetable} could have been considered for the case $p=1$.
However, as this would not have been feasible for $p=2,3$, we restricted $M$ below 9.
The rate for particular values of $M$, $s$, and $p$ in Table~\ref{tab:gradetable} was obtained by 
first computing the solution using the $\isoTD(M,r)$ index set, where $r$ is the smallest value for which $|\isoTD(M,r)| > 3 \cdot 10^3$, after which the rate is computed as in the previous examples using Legendre coefficients from the 10th largest up to the 100th largest. The largest Legendre coefficients were excluded for stability reasons.

The values in Table~\ref{tab:gradetable} agree relatively well with the related theory:
the rates corresponding to $s=6,8$ with $M=8$ and $p=1,2,3$ are close to the expected values of $5$ and $7$, respectively.
For the cases $s=2,4$, the obtained rates of approximately $r_2=2$ and $r_4=3.5$ are slightly better than the expected values of $1$ and $3$ predicted by~\cite{Cohen11}.
However, by increasing $M$ the rates keep approaching the predicted values: for example, $M=50$ gives rates $r_2=1.66$ and $r_4=3.26$ (not shown), respectively, and hence we expect that by further increasing $M$ the rates would eventually approach $1$ and $3$, respectively.
In summary, the larger $s$ is, the faster we expect the converge rate to approach $s-1$ when $M$ is increased.

According to Table~\ref{tab:gradetable}, the power $p$ does not essentially affect the rate.
This is expected: if we expanded the diffusion coefficient, the terms resulting from increasing the value of $p$ have smaller overall effect on the diffusion coefficient.
However, as the estimated rate depends on the actual Legendre coefficients used in the computation of the rate,
we have plotted the ordered Legendre coefficients corresponding to the last row of Table~\ref{tab:gradetable}, i.e., $M=8$ to confirm the results: 
Figure~\ref{fig:L2rates} shows the corresponding four sets ($s=2,4,6,8$) of almost overlapping curves ($p=1,2,3$).
As the ordered Legendre coefficients have the same profile for each fixed value of $s$, we conclude that the power $p$ does not (essentially) affect the asymptotic rate.

By ignoring the power $p$ and the type of the diffusion coefficient, we know from Sections~\ref{sec:n1}--\ref{sec:n3} that the values 
$M=2,4,6$ with $s=3/2$ correspond to the rates $r_{3/2}=2.27, 2.03$, and $1.43$, respectively.
These numbers are in line with the results presented in Table~\ref{tab:gradetable} for the case $s=2$.
By increasing $M$ the rate keeps decreasing: for example, $M=50$ gives rate $r_{3/2}=1.29$ (not shown), and hence we expect that by further increasing $M$ the rate would eventually approach $1/2$.

\begin{table}
\rowcolors{1}{}{lightgray}
\begin{tabular}[]{ccccc|ccccc|ccccc}
\multicolumn{5}{c}{$p=1$}&
\multicolumn{5}{c}{$p=2$}&
\multicolumn{5}{c}{$p=3$}\\
$M$&$s$=2&4&6&8&$M$&$s$=2&4&6&8&$M$&$s$=2&4&6&8\\
\hline
2&7.2&9.3&11.0&12.5&
2&6.1&8.0&9.6&11.0&
2&5.1&6.8&8.4&9.6\\
4&2.9&4.5&5.8&7.2&
4&2.6&4.0&5.4&6.8&
4&2.3&3.7&5.1&6.6\\
6&2.2&3.7&5.2&6.9&
6&2.1&3.5&5.1&6.8&
6&1.9&3.4&5.1&6.8\\
8&2.0&3.5&5.1&6.8&
8&1.9&3.5&5.1&6.8&
8&1.9&3.4&5.0&6.7
\end{tabular}
\caption{
Legendre coefficient convergence rates for the space-dependent diffusion coefficient~\eqref{equ:spacedepform} with different values of $M$, $s$, and $p$.
With $M=8$ and $p=1,2,3$ the expanded form of the diffusion coefficient has 9, 45, and 165 terms, respectively.
}
\label{tab:gradetable}
\end{table}

\begin{figure}
  \includegraphics{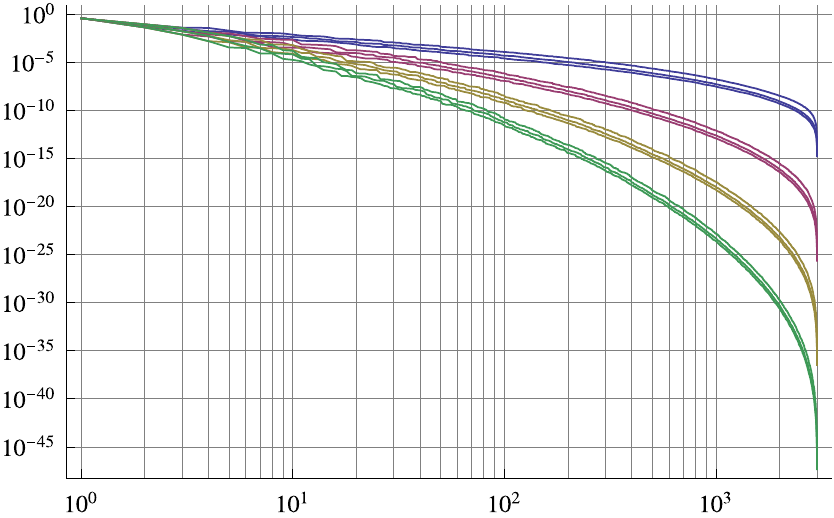}
  \caption{
    The ordered Legendre coefficient norms (vertical axis) corresponding to the last row of Table~\ref{tab:gradetable}, i.e., $M=8$.
    The four sets of three almost overlapping lines correspond to the four and three different values of $s$ and $p$, respectively, in Table~\ref{tab:gradetable}.
   }
  \label{fig:L2rates}
\end{figure}

\section{Discussion}
\label{sec:co}

In this paper, we have presented algorithms for constructing stochastic moment matrices of the form \eqref{equ:SMM} in an efficient manner for certain index set types allowing the modeler to consider non-affine diffusion coefficients of the form \eqref{equ:nonaffinekl}.

Algorithms~\ref{alg:multi1} and \ref{alg:Gmatrices} can be used (with possible slight modifications) to the efficient construction of hierarchical index sets together with the neighbour matrices required in the stochastic moment matrix generation.
The algorithms can be used with index sets which can be constructed by adding an initial multi-index to the index set, after which all additions to the index set are obtained by applying a certain kind of one-increment to a multi-index already in the index set. 
Most common index sets can be constructed in this way, with the exception of, for example, the factorial corrected total degree ($\TDFC$) index set.
We leave the generalization of Algorithms~\ref{alg:multi1} and \ref{alg:Gmatrices} to index set types where any one-increment is allowed for future studies.

In the numerical examples, we reported the performance of various index set types for 
one space-independent and two space-dependent 
diffusion coefficients in a one-dimensional spatial setting.
Moreover, we compared the observed Legendre coefficient convergence rates to theoretical results:
the observed Legendre coefficient rates agreed well with the theoretical results in~\cite{Cohen11}.

Using the isotropic total degree index set to compute the dimension weights for the anisotropic index sets as in \cite{Beck11,Beck12} worked well; the anisotropic index sets outperformed the isotropic index sets in all considered examples.
In the first experiment, the anisotropic tensor product index set gave the best results, whereas in the second and third experiment the anisotropic total degree index set performed the best.
Hence, it might be difficult to predict a priori which index set type would be the best choice without a possibly extensive analysis of the problem in question. 
According to our numerical experiments, a promising heuristic method to select the index set type is to first compute sGFEM solutions for various index set types and then select the one with the slowest growing ``tail'' of Legendre coefficient norms.
Other possibility is to 
compare the sGFEM solutions, in some suitable sense, with a set of, for example, FEM solutions corresponding to different realizations of the diffusion coefficient and then select the best performing index set type. 
The (numerical) consideration of these ideas is left for future studies.

\appendix
\section{Legendre and Hermite polynomials}
\label{sec:orthopolys}

When constructing the stochastic system matrices for \eqref{equ:reslinearsystem1} or \eqref{equ:reslinearsystem2} in the setting of generalized polynomial chaos (gPC) expansions, see, e.g., \cite{Xiu10}, we often encounter products of orthogonal hypergeometric polynomials. 
It turns out useful to linearize these products using a suitable polynomial family, i.e., solve the (generalized) linearization problem for the polynomial families involved.
Moreover, it is also useful to solve the inversion problem for each of the polynomial families, i.e., express $y^k$, where $k$ is a nonnegative integer, using the polynomial family in question; see \cite{Artes98} and the references therein for more information about linearizing expressions containing (hypergeometric) polynomials.

In this paper, we consider only the Legendre and Hermite polynomial families as those are typical choices in the sGFEM setting.
Hence, we consider the linearization and inversion problem only for these (hypergeometric) orthogonal polynomial families. However, similar results to the ones shown below can be obtained for other (hypergeometric) orthogonal polynomial families~\cite{Artes98}. 

In the next two sections, we recall the linearization, triple integral, and inversion formulas for the Legendre and Hermite polynomials.
We note that even though the presented results are classical, the actual formulas are rarely seen in the sGFEM literature.

\subsection{Legendre polynomials}

The orthonormal univariate Legendre polynomials are given explicitly by the Rodrigues' formula:
\begin{definition}[Legendre polynomials]
Let $m\in\N_0$.
The $m$th orthonormal univariate Legendre polynomial is given by
\[
L_m(y) = \frac{\sqrt{2m+1}}{2^m\,m!}\frac{\dd^m}{\dd y^m}[(y^2-1)^m]
\]
with the orthonormality relation
\[
\int_{-1}^1 L_k(y) L_l(y) \frac{1}{2} \dd y  = \delta_{kl},
\]
where $\delta_{kl}$ is the Kronecker's delta.
\end{definition}

\begin{theorem}[Linearization formula for the Legendre polynomials]
\label{thm:llp}
The product of two univariate Legendre polynomials $L_k(y)$ and $L_l(y)$ can be expanded as a sum of Legendre polynomials:
\begin{align}
\begin{split}
\label{equ:productofL1}
L_k(y)L_l(y) = \sum_{m=|k-l|}^{k+l}\ind{2g~\textrm{is even}}&\frac{A(g-k)A(g-l)A(g-m)}{A(g)}\\
&\times\frac{\sqrt{(2k+1)(2l+1)(2m+1)}}{2g+1}L_m(y),
\end{split}
\end{align}
where $g = (k+l+m)/2$ and $A(n) = {2n \choose n}/2^n$.
\end{theorem}
\begin{proof}
See~\cite{Adams78}.
\end{proof}

The coefficient in the series \eqref{equ:productofL1} can be written in a simpler form by employing the 3jm-symbol.
We need only the specific case with zero magnetic quantum numbers; see \cite{Edmonds57} for more information about the 3jm-symbols.

\begin{definition}[3jm-Symbol]
The 3jm-symbol with zero magnetic quantum numbers and $k,l,m\in\N_0$ is defined as
\label{equ:wignerdef}
\[
\wigner{k}{l}{m}^2 = \frac{(2g-2k)!(2g-2l)!(2g-2m)!}{(2g+1)!}\left[\frac{g!}{(g-k)!(g-l)!(g-m)!}\right]^2,
\]
if $2g$ is even and $|k-l|\leq m \leq k + l$; otherwise $\wignersmall{k}{l}{m}^2$ is zero. Here, $g$ is as in Theorem~\ref{thm:llp}.
\end{definition}

By noticing that 
\[
\ind{2g~\textrm{is even}} 
\ind{|k-l|\leq m \leq k+l}
\frac{A(g-k)A(g-l)A(g-m)}{A(g)}\frac{1}{2g+1}
= \wigner{k}{l}{m}^2,
\]
we can write \eqref{equ:productofL1} as 
\begin{align}
\label{equ:legendreprod}
L_k(y)L_l(y) = \sum_{m=|k-l|}^{k+l}L(k,l,m)L_m(y),
\end{align}
where
\[
L(k,l,m) = \sqrt{(2k+1)(2l+1)(2m+1)}\wigner{k}{l}{m}^2.
\]

Using the linearization formula~\eqref{equ:legendreprod}, we can compute the triple product integral of the Legendre polynomials.
\begin{theorem}[Triple product integral of the Legendre polynomials]
\label{thm:tintforLegendre}
The triple product integral of univariate Legendre polynomials $L_k(y), L_l(y),$ and $L_m(y)$ is given by
\begin{align*}
\int_{-1}^1 L_k(y)L_l(y)L_m(y)\rho(y)\dd y  
=L(k,l,m),
\end{align*}
where
\[
L(k,l,m) = \sqrt{(2k+1)(2l+1)(2m+1)}\wigner{k}{l}{m}^2.
\]
\end{theorem}
\begin{proof}
\begin{align*}
\int_{-1}^1 L_k(y)L_l(y)L_m(y)\rho(y)\dd y  
=\sum_{n=|k-l|}^{k+l}L(k,l,n)\delta_{nm}
=L(k,l,m),
\end{align*}
where we have first used \eqref{equ:legendreprod} and the orthonormality of the Legendre polynomials and then the selection rule $|k-l|\leq m \leq k + l$ of $L(k,l,m)$.
\end{proof}
Finally, we give the inversion formula for the univariate Legendre polynomials.
\begin{theorem}[Inversion formula for the Legendre polynomials]
\label{thm:inversioforLegendre}
The solution to the inversion problem for the univariate Legendre polynomials is given by
\[
y^k = \sum_{n=0}^k l^k_n L_n(y),
\]
where
\[
l^k_n = 
\ind{k-n~\textrm{is even}}
{k \choose n} 
n!\,
\sqrt{2n+1}
\frac{(k-n-1)!!}{(k+n+1)!!},
\quad 0 \leq n \leq k,
\]
and $!!$ denotes the double factorial.
\end{theorem}
\begin{proof}
The result can be obtained, e.g., by using the techniques in~\cite{Artes98}.
\end{proof}

\subsection{Hermite polynomials}
We start by defining the Hermite polynomials and then give the linearization, triple product integral, and inversion formulas for them.
\begin{definition}[Hermite polynomials]
Let $m\in\N_0$. The $m$th orthonormal (probabilists') univariate Hermite polynomial is given by 
\[
H_m(y) = \frac{(-1)^m}{\sqrt{m!}}\exp(y^2/2)\frac{\dd^m}{\dd y^m}\exp(-y^2/2),\qquad m=0,1,\ldots~,
\]
with the orthonormality relation
\[
\int_{\R} H_k(y) H_l(y) \frac{\exp(-y^2/2)}{\sqrt{2\pi}} \dd y  = \delta_{kl}.
\]
\end{definition}

\begin{theorem}[Linearization formula for the Hermite polynomials]
\label{thm:lhp}
The series for the product of two univariate Hermite polynomials $H_k(y)$ and $H_l(y)$ is given by
\begin{align*}
H_{k}(y)H_{l}(y) = 
\sum_{m=|k-l|}^{k+l}
H(k,l,m)H_m(y),
\end{align*}
where 
\[
H(k,l,m) = \ind{2g~\textrm{is even}}\ind{|k-l|\leq m \leq k+l}\left[{k \choose {g-m} } {l \choose {g-m} } {m \choose {g-k} }\right]^\frac{1}{2}
\]
with $g = (k+l+m)/2$.
\end{theorem}
\begin{proof}
See~\cite{Artes98}.
\end{proof}

By using Theorem~\ref{thm:lhp}, we obtain the triple product integral of the Hermite polynomials in a similar way as in the Legendre case.
\begin{theorem}[Triple product integral of the Hermite polynomials]
\label{thm:tintforHermite}
The triple product integral of univariate Hermite polynomials $H_k(y), H_l(y),$ and $H_m(y)$ is given by
\begin{align*}
\int_\R H_k(y) H_l(y) H_m(y) \rho(y) \dd y 
=H(k,l,m).
\end{align*}
\end{theorem}
We conclude with the solution to the inversion problem for the Hermite polynomials.
\begin{theorem}[Inversion formula for the Hermite polynomials]
\label{thm:inversioforHermite}
\[
y^k = \sum_{n=0}^k h^k_n H_n(y),
\]
where
\[
h^k_n = 
\ind{k-n~\textrm{is even}}
{k \choose n} 
\sqrt{n!}\,
(k-n-1)!!,
\quad 0 \leq n \leq k.
\]
\end{theorem}
\begin{proof}
See~\cite{Artes98}.
\end{proof}
\clearpage
\section{ \texorpdfstring{Pseudocode for generating the $\mathtt{TS}$ index set}{Pseudocode for generating the TS index set} }
\label{App:A}
\setcounter{algocf}{2}
\begin{algorithm}[H]
  \caption{Pseudocode for generating the index set $\TS(\mu,\varepsilon)$ and the corresponding parenthood relations $\PtC$.}
  \label{alg:multi1}
 \SetLine 
 \KwData{Vector $\mu \in c_0$ and real number $\varepsilon\in (0,1)$.}
 \KwResult{Multi-indices $\TS(\mu,\varepsilon)$ and parenthood relations $\PtC$.}
 initialize $\alpha$ to empty vector of integer pairs\;
 push pair $(\alpha,1)$ to the end of $\TS(\mu,\varepsilon)$\;
 initialize $\varepsilon_\alpha$, $m$, $p$, and $c$ to one\;
 initialize $n$ to two\;
 create empty stacks for $\alpha$, $m$, $\varepsilon_{\alpha}$, and $p$\;
 \While{True}{
   \If{$\varepsilon_{\alpha}\mu_m\geq\varepsilon$}{
     push $\alpha$, $m$, and $\varepsilon_{\alpha}$ to the corresponding stack\;
     push $c$ to $p$-stack\; 
     increase $m$ by one\;
     continue\;
   }
     \lIf{$m$-stack is empty}{return $\TS(\mu,\varepsilon)$ and $\PtC$}
     set $\alpha$, $m$, $\varepsilon_{\alpha}$, and $p$ to the top value in the corresponding stack\;
     pop $\alpha$, $m$, $\varepsilon_{\alpha}$, and $p$-stacks\;
     
     \lIf{\# of pairs in $\alpha = 0$ OR $\alpha[-1][1]\neq m$}{
       push integer pair $(m,0)$ to the end of $\alpha$
     }

     increase $\alpha[-1][2]$ by one\;
     multiply $\varepsilon_{\alpha}$ with $\mu_m$\;
   set $c$ to $n$\;
   increase $n$ by one\;
   push pair $(\alpha, \varepsilon_{\alpha})$ to the end of $\TS(\mu,\varepsilon)$\;
   \lWhile{\# of elements in $\PtC < p$}{push an empty integer vector to the end of $\PtC$}
   push $c$ to the front of $\PtC[p]$\;
 }
\end{algorithm}

\section{Example run of Algorithm~\ref{alg:multi1}}
\label{App:examplerun}
An example run of Algorithm~\ref{alg:multi1} with intermediate steps in the spirit of Algorithm~\ref{alg:generalTS} is shown in Figure~\ref{fig:TSexample1a} and in Table~\ref{tbl:TSexample1b} for 
\[
\mu=(2^{-1}, 3^{-1}, 4^{-1}, 50^{-1}, 60^{-1}, 70^{-1}, \ldots)\,\textrm{ and }\,\varepsilon = 20^{-1}.
\]

Figure~\ref{fig:TSexample1a} shows both the resulting index set $\TS(\mu,\varepsilon)$ and the corresponding $\PtC$ information, and Table~\ref{tbl:TSexample1b} contains the intermediate steps, i.e., for each pass through the main loop in the algorithm, the most recent addition to the index set, the stack just before the next addition to the index set, and the considered child which was not added to the stack.

\begin{landscape}
\vspace*{\fill}
\begin{figure}[htb!]
        \includegraphics{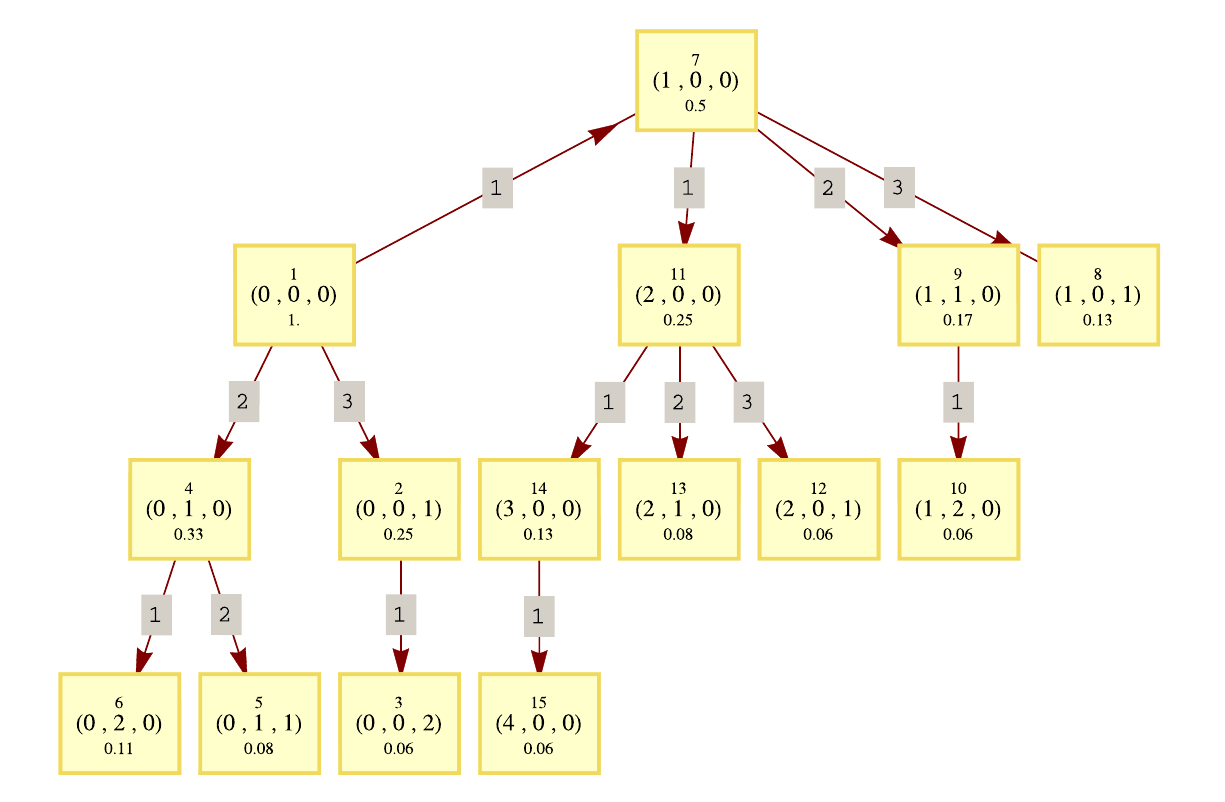}
	\caption{Index set $\TS(\mu,\varepsilon)$ and the corresponding $\PtC$ information for
          $\mu=(2^{-1}, 3^{-1}, 4^{-1}, 50^{-1}, 60^{-1}, 70^{-1}, \ldots)$
          and $\varepsilon = 20^{-1}$.
          The vertices correspond to the multi-indices in $\TS(\mu,\varepsilon)$.
          The vertex labels show the generation order (top),
          the first three indices (the rest are zeros) (middle),
          and the (rounded) value $\mu^\alpha$  (bottom) of the multi-indices.
          The directed edges correspond to the parenthood relations between the multi-indices, i.e., they display the information stored in $\PtC$.
          The edge labels indicate the order of the children of, say, $\alpha$ in the vector $\PtC[\alpha]$.
}
	\label{fig:TSexample1a}
\end{figure}
\vspace*{\fill}
\end{landscape}

\begin{landscape}
\vspace*{\fill}
\begin{table}[htb!]
\Large
\rowcolors{1}{}{lightgray}
\begin{tabular}[]{cclc|cclc}
\#&Addition&Stack&Considered&\#&Addition&Stack&Considered\\
\hline
1&
$\mind{(0,0,0)}{1}$&
$\scriptstyle{(1,0,0),\,(0,1,0),\,(0,0,1)}$&
$\mind{(0,0,0,1)}{0.02}$&
9&
$\mind{(1,1,0)}{0.17}$&
$\scriptstyle{(2,0,0),\,(1,2,0)}$&
$\mind{(1,1,1,0)}{0.04}$
\\
2&
$\mind{(0,0,1)}{0.25}$&
$\scriptstyle{(1,0,0),\,(0,1,0),\,(0,0,2)}$&
$\mind{(0,0,1,1)}{0.01}$&
10&
$\mind{(1,2,0)}{0.06}$&
$\scriptstyle{(2,0,0)}$&
$\mind{(1,3,0,0)}{0.02}$
\\
3&
$\mind{(0,0,2)}{0.06}$&
$\scriptstyle{(1,0,0),\,(0,1,0)}$&
$\mind{(0,0,3,0)}{0.02}$&
11&
$\mind{(2,0,0)}{0.25}$&
$\scriptstyle{(3,0,0),\,(2,1,0),\,(2,0,1)}$&
$\mind{(2,0,0,1)}{0.01}$
\\
4&
$\mind{(0,1,0)}{0.33}$&
$\scriptstyle{(1,0,0),\,(0,2,0),\,(0,1,1)}$&
$\mind{(0,1,0,1)}{0.01}$&
12&
$\mind{(2,0,1)}{0.06}$&
$\scriptstyle{(3,0,0),\,(2,1,0)}$&
$\mind{(2,0,2,0)}{0.02}$
\\
5&
$\mind{(0,1,1)}{0.08}$&
$\scriptstyle{(1,0,0),\,(0,2,0)}$&
$\mind{(0,1,2,0)}{0.02}$&
13&
$\mind{(2,1,0)}{0.08}$&
$\scriptstyle{(3,0,0)}$&
$\mind{(2,2,0,0)}{0.03}$
\\
6&
$\mind{(0,2,0)}{0.11}$&
$\scriptstyle{(1,0,0)}$&
$\mind{(0,3,0,0)}{0.04}$&
14&
$\mind{(3,0,0)}{0.13}$&
$\scriptstyle{(4,0,0)}$&
$\mind{(3,1,0,0)}{0.04}$
\\
7&
$\mind{(1,0,0)}{0.5}$&
$\scriptstyle{(2,0,0),\,(1,1,0),\,(1,0,1)}$&
$\mind{(1,0,0,1)}{0.01}$&
15&
$\mind{(4,0,0)}{0.06}$&
&
$\mind{(5,0,0,0)}{0.03}$
\\
8&
$\mind{(1,0,1)}{0.13}$&
$\scriptstyle{(2,0,0),\,(1,1,0)}$&
$\mind{(1,0,2,0)}{0.03}$\\
\end{tabular}
\caption{
Intermediate steps in Algorithm~\ref{alg:multi1} when constructing the index set $\TS(\mu,\varepsilon)$ for 
$\mu=(2^{-1}, 3^{-1}, 4^{-1}, 50^{-1}, 60^{-1}, 70^{-1}, \ldots)$
and $\varepsilon = 20^{-1}$.
From left to right, the columns consists of
the number of times we have passed through the start of the main loop,
the first three indices (the rest are zeros) of the most recently added multi-index to $\TS(\mu,\varepsilon)$ with the corresponding (rounded) value $\mu^\alpha$ shown below,
multi-indices in the stack just before the next addition to the index set, 
and the considered child which was not added to the stack.
\label{tbl:TSexample1b}
}
\end{table}
\vspace*{\fill}
\end{landscape}

\section{Pseudocode for generating the neighbour matrices}
\label{App:B}
\begin{algorithm}[H]
  \caption{Pseudocode for generating the neighbour matrices $\{N^w\}_{w \in W_\Xi}$ corresponding to the index set $\TS(\mu,\varepsilon)$.}
  \label{alg:Gmatrices}
  \SetLine  
  \KwData{Weights $W_\Xi$, index set $\TS(\mu,\varepsilon)$, and the corresponding vector $\mu$, real number $\varepsilon$, and parenthood relations $\PtC$.}
  \KwResult{Neighbour matrices $\{N^w\}_{w \in W_\Xi}$.}
  \For{$m = 1$ to \# of elements in $W_\Xi$}{
    set $w$ to $W_\Xi[m]$\;
    set $\varepsilon_w$ to $\mu^w$\;
    \For{$n = 1$ to \# of elements in $\TS(\mu,\varepsilon)$}{
      set $\eta$ to $\TS(\mu,\varepsilon)[n][1]$\;
      set $\varepsilon_{\eta}$ to $\TS(\mu,\varepsilon)[n][2]$\;
      set $\gamma$ to $\eta - w$\;
      \lIf{$\gamma[r][2] < 0$ for some $r$ OR $\varepsilon_{\eta} / \varepsilon_{w} < \varepsilon $}{continue}
      initialize $k$ and $u$ to one\;
      \For{each pair $p$ in $\gamma$ in the increasing position order}{
        \lIf{$p[2] = 0$}{continue}
        initialize $j$ to $p[2]$\;	  
        \If{$p[1] > u$}{
          decrease $j$ by one\;
          set $k$ to $\PtC[k][p[1]-u+1]$\;
          set $u$ to $p[1]$\;
        }	  
        \While{$j > 0$}{
          set $k$ to $\PtC[k][1]$\;
          decrease $j$ by one\;
        }
      }
      set $\big[N^w\big]_{n, k}$ and $\big[N^w\big]_{k, n}$ to one\;
    }
  }
  return $\{N^w\}_{w \in W_\Xi}$\;
\end{algorithm}

\section{ \texorpdfstring{Modifications for the $\mathtt{isoTD}$ index set}{Modifications for the isoTD index set} }
\label{App:C}
\begin{algorithm}[H]
\nocaptionofalgo
\caption{Modifications to Algorithm~\ref{alg:multi1} for the $\isoTD$ index set.}
\KwData{Positive integer $N$ and nonnegative integer $K$.}
\SetLine 
\setcounter{AlgoLine}{1}
push pair $(\alpha,0)$ to the end of $\TD(N,K)$\;
initialize $\varepsilon_\alpha$ to zero; initialize $m$, $p$, and $c$ to one\;
\setcounter{AlgoLine}{6}
\lIf{$\varepsilon_{\alpha} < K$ AND $m \leq N$}{\ldots}
\setcounter{AlgoLine}{17}
increase $\varepsilon_{\alpha}$ by one\;
\end{algorithm}

\begin{algorithm}[H]
\nocaptionofalgo
\caption{Modifications to Algorithm~\ref{alg:Gmatrices} for the $\isoTD$ index set.}
\KwData{Weights $W_\Xi$, index set $\isoTD(N,K)$, and the corresponding $N$, $K$, and $\PtC$.}
\SetLine 
\setcounter{AlgoLine}{2}
set $\varepsilon_w$ to $\sum_{r\geq1}w_r$\;
\setcounter{AlgoLine}{7}
\lIf{$\length{w} > N$ OR $\gamma[r][2] < 0$ for some $r$ OR $\varepsilon_{\eta} - \varepsilon_{w} > K $}{continue}
\end{algorithm}

\section{ \texorpdfstring{Modifications for the $\mathtt{aTP}$ index set}{Modifications for the aTP index set} }
\label{App:D}

\begin{algorithm}[H]
\nocaptionofalgo
\caption{Modifications to Algorithm~\ref{alg:multi1} for the $\aTP$ index set.}
\KwData{Positive integer $N$, nonnegative integer $K$, and weight vector $\mu\in\R_+^N$.}
\SetLine 
\setcounter{AlgoLine}{1}
push pair $(\alpha,0)$ to the end of $\aTP(N,K,\mu)$\;
initialize $\varepsilon_\alpha$ to zero; initialize $m$, $p$, and $c$ to one\;
\setcounter{AlgoLine}{6}

\lIf{$\varepsilon_\alpha \leq K$ AND $m \leq N$}{\ldots}
\setcounter{AlgoLine}{9}
increase $m$ by one; set $\varepsilon_\alpha$ to zero\;
\setcounter{AlgoLine}{12}
\While{True}{
\setcounter{AlgoLine}{12}
\lIf{$m$-stack is empty}{return $\aTP(N,K,\mu)$ and $\PtC$}
\setcounter{AlgoLine}{17}
increase $\varepsilon_{\alpha}$ by $\frac{\mu_m}{\mu_{\min}}$\;
\setcounter{AlgoLine}{17}
\lIf{$\varepsilon_\alpha \leq  K $}{break}
\setcounter{AlgoLine}{17}
}
\end{algorithm}
Notice that the lines 13 and 18 in Algorithm~\ref{alg:multi1} are replaced with multiple lines.

\begin{algorithm}[H]
\nocaptionofalgo
\caption{Modifications to Algorithm~\ref{alg:Gmatrices} for the $\aTP$ index set}
\KwData{Weights $W_\Xi$, multi-indices $\aTP(N,K,\mu)$, and the corresponding $N$, $K$, $\mu$, and $\PtC$.}
\SetLine 
\setcounter{AlgoLine}{2}
{\em remove this line}\;
\setcounter{AlgoLine}{7}
\lIf{$\length{w} > N$ OR $\gamma[r][2] < 0$ OR $\frac{\mu_{\gamma[r][1]}}{\mu_{\min}} \gamma[r][2] > K$ for some $r$}{continue}
\setcounter{AlgoLine}{14}
set $k$ to $\PtC[k][-1-\aTP(N,K,\mu)[\PtC[k][-1]][1][-1][1]+p[1]]$\;
\end{algorithm}
Here, $\PtC[\alpha][-m]$ denotes the $m$th element of the vector $\PtC[\alpha]$ counting from the end.
In the $\aTP$ case it might happen that even though the multi-indices
$(\alpha_1,\ldots,\alpha_m,0,0,\ldots)$ and
$(\alpha_1,\ldots,\alpha_m,1,0,\ldots)$
belong to the $\aTP$ index set the multi-index
$(\alpha_1,\ldots,\alpha_m+1,0,0,\ldots)$ does not.
Due to this fact, we had to introduce a new while loop between lines 13 and 18 of Algorithm~\ref{alg:multi1} to ensure that only valid multi-indices are added to the index set.
Moreover, we had to process the position-value pairs in a different manner than in the other cases resulting in the difficult looking line 15 in Algorithm~\ref{alg:Gmatrices}. On that line, the one step down the tree to the child with the length $m$ is done from the end and not from the front.

If the dimension weights $\mu$ are given using the vector $g$, we need to replace 
$\frac{\mu_m}{\mu_{\min}}$ on line~18 of the $\aTP$ modified Algorithm~\ref{alg:multi1} with $g_m$ and 
$\frac{\mu_{\gamma[r][1]}}{\mu_{\min}}$ on line~8 of the $\aTP$ modified Algorithm~\ref{alg:Gmatrices} with $g_{\gamma[r][1]}$
and we consider $K$ as a nonnegative real number.

\bibliographystyle{acm}

\bibliography{Hakula_Leinonen_On_efficient_construction_of_stochastic_moment_matrices}

\begin{thebibliography}{10}

\bibitem{Adams78}
{\sc Adams, J.~C.}
\newblock On the expression of the product of any two {L}egendre's coefficients
  by means of a series of {L}egendre's coefficients.
\newblock {\em Proc. R. Soc. Lond. 27\/} (1878), 63--71.

\bibitem{Adler07}
{\sc Adler, R.~J., and Taylor, J.~E.}
\newblock {\em Random Fields and Geometry}, vol.~115 of {\em Springer
  monographs in mathematics}.
\newblock Springer, New York, 2007.

\bibitem{Artes98}
{\sc Art\'es, P.~L., Dehesa, J.~S., Mart\'inez-Finkelshtein, A., and
  S\'anchez-Ruiz, J.}
\newblock Linearization and connection coefficients for hypergeometric-type
  polynomials.
\newblock {\em J. Comput. Appl. Math. 99}, 1 (1998), 15--26.

\bibitem{Babuska10}
{\sc Babu\v{s}ka, I.~M., Nobile, F., and Tempone, R.~F.}
\newblock A stochastic collocation method for elliptic partial differential
  equations with random input data.
\newblock {\em SIAM J. Numer. Anal. 45}, 3 (2007), 1005--1034.

\bibitem{Back11}
{\sc B\"ack, J., Nobile, F., Tamellini, L., and Tempone, R.~F.}
\newblock Stochastic spectral galerkin and collocation methods for {PDE}s with
  random coefficients: A numerical comparison.
\newblock In {\em Spectral and High Order Methods for Partial Differential
  Equations}, J.~S. Hesthaven and E.~M. Rønquist, Eds., vol.~76 of {\em Lect.
  Notes Comput. Sci. Eng.} Springer, Heidelberg, 2011, pp.~43--62.

\bibitem{Beck11}
{\sc Beck, J., Nobile, F., Tamellini, L., and Tempone, R.~F.}
\newblock Implementation of optimal {G}alerkin and collocation approximations
  of {PDE}s with random coefficients.
\newblock {\em ESAIM: Proc. 33\/} (2011), 10--21.

\bibitem{Beck12}
{\sc Beck, J., Tempone, R.~F., Nobile, F., and Tamellini, L.}
\newblock On the optimal polynomial approximation of stochastic {PDE}s by
  {G}alerkin and collocation methods.
\newblock {\em Math. Models Methods Appl. Sci. 22}, 9 (2012), 1250023 (33 pp.).

\bibitem{BieriThesis09}
{\sc Bieri, M.}
\newblock {\em Sparse Tensor Discretizations of Elliptic PDEs with Random Input
  Data}.
\newblock PhD thesis, ETH Z\"urich, 2009.
\newblock Diss. ETH No. 18598.

\bibitem{Bieri09a}
{\sc Bieri, M., Andreev, R., and Schwab, {\relax Ch}.}
\newblock Sparse tensor discretization of elliptic s{PDE}s.
\newblock {\em SIAM J. Sci. Comput. 31}, 6 (2009), 4281--4304.

\bibitem{Bieri09b}
{\sc Bieri, M., Andreev, R., and Schwab, {\relax Ch}.}
\newblock Sparse tensor discretization of elliptic s{PDE}s.
\newblock Research Report 2009-07, Seminar for Applied Mathematics, ETH
  Z{\"u}rich, 2009.
\newblock
  \url{http://www.sam.math.ethz.ch/sam_reports/reports_final/reports2009/2009-%
07.pdf}.

\bibitem{Bieri09c}
{\sc Bieri, M., and Schwab, {\relax Ch}.}
\newblock Sparse high order {FEM} for elliptic s{PDE}s.
\newblock {\em Comput. Methods Appl. Mech. Engrg. 198}, 13--14 (2009),
  1149--1170.

\bibitem{Charrier12}
{\sc Charrier, J.}
\newblock Strong and weak error estimates for elliptic partial differential
  equations with random coefficients.
\newblock {\em SIAM J. Numer. Anal. 50}, 1 (2012), 216--246.

\bibitem{Chkifa15}
{\sc Chkifa, A., Cohen, A., and Schwab, {\relax Ch}.}
\newblock Breaking the curse of dimensionality in sparse polynomial
  approximation of parametric {PDE}s.
\newblock {\em J. Math. Pures Appl. 103}, 2 (2015), 400--428.

\bibitem{Cohen11}
{\sc Cohen, A., DeVore, R., and Schwab, {\relax Ch}.}
\newblock Analytic regularity and polynomial approximation of parametric and
  stochastic elliptic {PDE}'s.
\newblock {\em Anal. Appl. 9}, 1 (2011), 11--47.

\bibitem{Edmonds57}
{\sc Edmonds, A.~R.}
\newblock {\em Angular Momentum in Quantum Mechanics}.
\newblock Princeton University Press, Princeton, N.J., 1957.

\bibitem{Gittelson10}
{\sc Gittelson, C.~J.}
\newblock Stochastic {G}alerkin discretization of the log-normal isotropic
  diffusion problem.
\newblock {\em Math. Mod. Meth. Appl. Sci. 20}, 2 (2010), 237--263.

\bibitem{Hyvonen14}
{\sc Hyv\"onen, N., and Leinonen, M.}
\newblock Stochastic {G}alerkin finite element method with local conductivity
  basis for electrical impedance tomography.
\newblock {\em arXiv:1412.1941\/} (2014), 21 pp.

\bibitem{Schwab11a}
{\sc Schwab, {\relax Ch}., and Gittelson, C.~J.}
\newblock Sparse tensor discretizations of high-dimensional parametric and
  stochastic {PDE}s.
\newblock {\em Acta Numer. 20\/} (2011), 291--467.

\bibitem{Vitter85}
{\sc Vitter, J.~S.}
\newblock Random sampling with a reservoir.
\newblock {\em ACM Trans. Math. Software 11}, 1 (1985), 37--57.

\bibitem{Yos15}
{\sc Weiss, D., and Yosibash, Z.}
\newblock Uncertainty quantification for a 1{D} thermo-hyperelastic coupled
  problem using polynomial chaos projection and $p$-{FEM}s.
\newblock {\em Submitted to: Comput. Math. Appl.\/} (2015).

\bibitem{Mathematica2012}
{\sc Wolfram~Research, {\relax Inc}.}
\newblock {\em Mathematica, Version 9.0}.
\newblock Wolfram Research, Inc., Champaign, Illinois, 2012.

\bibitem{Xiu10}
{\sc Xiu, D.}
\newblock {\em Numerical Methods for Stochastic Computations: A Spectral Method
  Approach}.
\newblock Princeton University Press, Princeton, N.J., 2010.

\end{thebibliography}

\end{document}